\documentclass[a4paper, twosided, 10pt,reqno]{amsart}
\usepackage[hidelinks]{hyperref}
\usepackage{amsthm,amsmath,amssymb,mathtools}
\usepackage{enumitem}
\usepackage[english]{babel}

\usepackage{enumitem}
\usepackage{microtype}
\usepackage{graphicx}
\usepackage{multicol}
\usepackage{tikz-cd}
\usetikzlibrary{patterns}
\usepackage{pgfplots}
\usetikzlibrary{intersections}

\newcommand*{\defeq}{\mathrel{\vcenter{\baselineskip0.5ex \lineskiplimit0pt
			\hbox{\scriptsize.}\hbox{\scriptsize.}}}%
	=}
\renewcommand{\epsilon}{\varepsilon}
\renewcommand{\theta}{\vartheta}
\renewcommand{\phi}{\varphi}
\newcommand{\R}{\mathbb{R}}

\newcommand{\N}{\mathbb{N}}

\newcommand{\esse}{\mathbb{S}}

\usepackage{accents}

\makeatletter

\DeclarePairedDelimiter\pe{\langle}{\rangle}

\title{Some structure theorems for Weingarten surfaces}
\author{Angelo Benedetti}

\address{Universit\`a degli Studi dell'Aquila \\
	Dipartimento di Ingegneria e Scienze dell'Informazione e Matematica\\
	via Vetoio 1 \\
	67100 L'Aquila\\
	Italy}
\email{angelo.benedetti@graduate.univaq.it}

\newtheorem{lemma}{Lemma}[section]
\newtheorem{theorem}[lemma]{Theorem}
\newtheorem{theoremletter}{Theorem}

\newtheorem{lemmaletter}[theoremletter]{Lemma}

\newtheorem{corollary}[lemma]{Corollary}

\theoremstyle{definition}
\newtheorem{definition}[lemma]{Definition}
\newtheorem*{definition*}{Definition}
\theoremstyle{remark}

\newtheorem*{remark*}{Remark}

\begin{document}
		\let\thefootnote\relax	
	\begin{abstract}
		Let \(M\subset\mathbb{R}^3\) be a properly embedded, connected, complete surface with boundary a convex planar curve \(C\), satisfying an elliptic equation \(H=f(H^2-K)\), where \(H\) and \(K\) are the mean and the Gauss curvature respectively -- which we will refer to as \textit{Weingarten equation}. When \(M\) is contained in one of the two halfspaces determined by \(C\), we give sufficient conditions for \(M\) to inherit the symmetries of \(C\). In particular, when \(M\) is vertically cylindrically bounded, we get that \(M\) is rotational if \(C\) is a circle. In the case in which the Weingarten equation is linear, we give a sufficient condition for such a surface to be contained in a halfspace. Both results are generalizations, to the Weingarten setting, of results of Rosenberg and Sa Earp for constant mean curvature surfaces. In particular, our results also recover and generalize the constant mean curvature case.
	\end{abstract}	
		\maketitle
		\section{Introduction}
	Weingarten surfaces in \(\R^3\) are a generalization of constant mean curvature (cmc) surfaces where, instead of the mean curvature \(H\) to be constant, we require a certain functional relation -- to which we will refer as \textit{Weingarten equation} -- between\,\(H\) and the Gauss curvature \(K\) to hold. In many cases, this may be expressed as: \[H=f(H^2-K),\] where \(f\) is a given (smooth) real-valued function defined on \([0,+\infty)\). This is the case, for example, of linear Weingarten surfaces, i.e. of surfaces satisfying\,\,\(2aH+bK=1\), for\,\(a\), \(b\) constants \cite{saearpweingarten}. A case of interest is when the Weingarten equation is elliptic. The reader may find some relevant references in \cite{britoboundarycircle, chernsphere,  chernwsurfaces, tenenblat, galvezlinear, galvezmira, hartmanwinter}. This list is by no means exhaustive. It is shown in \cite{saearpweingarten} that ellipticity is equivalent to requiring that: \[4tf'(t)^2<1 \text{ for each } t\geq 0.\] Note that this condition is satisfied by the linear Weingarten surfaces with \(a^2+b>0\). In the elliptic case, many results of the theory of cmc surfaces are seen to hold. In fact, ellipticity permits to apply some tools of cmc theory, e.g. Alexandrov reflection technique. For example, the theory of Meeks \cite{meeks} and Korevaar, Kusner and Solomon \cite{Korevaar1989TheSO} is seen to hold for a wide class of Weingarten equations \cite{saearpweingarten} \cite{espinar}.

	In this spirit, this paper originated mainly as a generalization to the setting of Weingarten surfaces of some results of Rosenberg and Sa Earp for cmc surfaces \cite{saearpconvex}. During the work, the application of techniques which differ from the ones in \cite{saearpconvex} led to results which are original also for cmc surfaces. The main object is complete, properly embedded surfaces satisfying an elliptic Weingarten relation, having their boundary on a planar curve \(C\). Without loss of generality, we assume that the curve \(C\) is contained in the horizontal plane \(\{z=0\}\). Moreover, we assume that the surface is transverse to the plane along its boundary \(C\).
	
	When \(C\) is a circle and \(M\) is compact, it is known that \(M\) is rotational -- i.e. it is a spherical cap \cite{britosaearpmeeks} \cite{saearpweingarten}. Observe that this is known only in the case in which \(M\) is assumed to be transverse to the plane $\{z=0\}$ along \(C\). What we do in the following is investigate the case when \(M\) is not compact. One would like to prove that, also in this case, \(M\) is rotational.
	
	In Theorem \ref{zero} we assume that, outside of a compact set, \(M\) is contained in a finite set of solid half cylinders, not necessarily vertical (see Figure \ref{figuretildem}). This is a reasonable assumption, as all finite type elliptic Weingarten surfaces, having \(f(0)\neq0\), satisfy this property \cite{saearpweingarten}. In particular, all rotational elliptic Weingarten surfaces, having \(f(0)\neq0\), are contained in a solid cylinder \cite{fernandezweingarten} \cite{saearpfrench}. When \(M\) is entirely contained in the halfspace \(\{z\geq0\}\), we show that the surface inherits some of the symmetries of \(C\). The result is sharp, in the sense that we exhibit an example of a surface which inherits only the symmetries of its boundary which are detected by the theorem. When \(C\) is a circle and \(M\) is entirely contained in a solid vertical cylinder, we get that \(M\) is the annular end of a rotational Weingarten surface, thereby recovering the analogy with the compact\,\,case. 
	
	The problem is then reduced to understand when the surface \(M\) is contained in one of the two halfspaces determined by \(\{z=0\}\), which is tackled in Theorem \ref{two} for linear Weingarten surfaces. When the surface is compact, this is always the case. In the case in which \(M\) is not compact, we find some sufficient conditions for the case in which \(M\) is vertically cylindrically bounded and it has a finite number of vertical annular ends. To the best of our knowledge, when \(C\) is a circle, no counterexample is known of a non-compact surface \(M\) satisfying the hypotheses of the theorem and not being an annular end of a rotational surface -- both in the Weingarten and cmc setting. Using the previous results, we then give a bound on the number of ends of such a Weingarten surface.

	We remark that the proof of Theorem \ref{two} leads also to a generalization of \cite[Theorem 2]{saearpconvex}, which is the analogous result in the cmc case. In fact, in \cite{saearpconvex} the surface is assumed to have only positive ends -- i.e. they diverge only in the upper halfspace -- while we allow them also to be negative, only requiring the surface to be vertically cylindrically bounded. This seems like a more natural condition, compared to the statement of Theorem \ref{zero}.
	
	\section{Preliminaries}\label{sectionpreliminaries}
	
	Throughout the paper, we use the language of Alexandrov reflection principle, that we introduce now. Let $\Sigma\subset\R^{n+1}$ be a closed, embedded, orientable hypersurface, for \(n\geq1\). First observe that $\Sigma$ separates the ambient space in two connected components. We refer as \textit{interior} to the bounded one and as \textit{exterior} to the unbounded one. Now, consider a hyperplane \(\pi\), disjoint from $\Sigma$, with fixed orthogonal direction $\nu$. Assume that $\nu$ points toward the halfspace containing $\Sigma$.
	
	Define the family of translations:
	\[\pi(t)\defeq \pi+t\nu,\]
	for \(t\in\R\), and the families of halfspaces:
	\[\Pi^-(t)\defeq\bigcup_{s\leq t}\pi_s, \qquad \Pi^+(t)\defeq\bigcup_{s\geq t}\pi_s.\]
	For a value of $t$ such that \(\Pi^-(t)\cap\Sigma\neq\emptyset\), define  $\Sigma^*_t$ as the reflection through $\pi(t)$ of $\Sigma\cap\Pi^-(t)$.
	
	By compactness, we know that there exists a maximal \(T>0\) such that\,\,\(\pi(t)\cap\Sigma=\emptyset\) for each \(t<T\). Then, since $\Sigma$ is closed, the hyperplane \(\pi(T)\) is tangent to $\Sigma$ at some point. Moreover, one can show that there exists a value of $\epsilon>0$ such that, for $\tau=T+\epsilon$, it holds:
	\begin{itemize}
		\item $\Sigma^*_{\tau}$ is contained in the interior of $\Sigma$;
		\item \(\Sigma^*_{\tau}\) is a graph over a region of \(\pi(t+\epsilon)\).
	\end{itemize}
	Again, since $\Sigma$ is bounded, there exists a value $\tilde{\tau}>T+\epsilon$ minimal such that one of the last two conditions doesn't hold for $\tau=\tilde{\tau}$. We then say that the Alexandrov procedure with direction $\nu$ \textit{stops} at the plane $\pi(\tilde{\tau})$.
	
	When $\Sigma$ satisfies, locally, an elliptic equation for which the interior and boundary maximum principle hold -- e.g. $\Sigma$ has constant mean curvature, or it satisfies an elliptic Weingarten equation -- it is known that a plane for which the Alexandrov procedure stops is a plane of symmetry for the surface.

	\begin{definition}
		Let \(\Sigma\subset\R^{n+1}\) be a closed orientable hypersurface. We say that $\Sigma$ has \textit{Alexandrov symmetry} in the direction $\nu\in\esse^{n}$ if the Alexandrov procedure in both direction $\nu$ and \(-\nu\) stops at a hyperplane \(Q_0\) which is the same for both directions.
	\end{definition}
	Observe that Alexandrov symmetry implies symmetry with respect to \(Q_0\). In fact, Alexandrov symmetry is the kind of symmetry which may be inferred from an application of the Alexandrov reflection principle.
	
	In both proofs we use the theory of Korevaar, Kusner and Solomon on the monotonicity of the Alexandrov function, which permits us to use the Alexandrov reflection principle in the presence of (vertical) annular ends. First, we have to introduce the language of Alexandrov functions. We do it with lesser generality than what is done in \cite{Korevaar1989TheSO}, only defining what we need in the following. Let \(M\) be a properly embedded complete elliptic Weingarten surface in \(\R^3\), with boundary on the plane \(\{z=0\}\). Assume\,\,\(M\) is contained in the halfspace \(\R^3_+\), it is vertically cylindrically bounded and it is transverse to the plane \(\{z=0\}\) along \(\partial M\). Let \(Q\) be a vertical plane with normal $\nu$. For a point \(p\in Q\), define the line\,\(L_p\defeq p+\R\nu\). Assume that the line \(L_p\) intersects \(M\). Now, for \(\tau>0\) sufficiently big,\,\,\(p+t\nu\) is disjoint from \(M\) for each \(t\geq\tau\). Let $t_1$ be the first time, as the parameter decreases from $\tau$, such that \(p+t_1\nu\in M\). Now, if \(L_p\) is tangent to \(M\) at \(p+t_1\nu\), set\,$\alpha_1(p)\defeq t_1$. Otherwise, there exists \(t_2\) first time, as the parameter decreases from \(t_1\), such that \(p+t_2\nu\in M\). In this case, set $\alpha_1(p)\defeq(t_1+t_2)/2$. The function\,$\alpha_1$ is called Alexandrov function. It was described in \cite{Korevaar1989TheSO}, where they used it to show their characterization of Delaunay surfaces.
	
	Additionaly, for \(t\geq 0\), one defines:
	\begin{equation}
		\alpha(t)\defeq\sup_{\pe{p,e_3}=t}\alpha_1(p),
	\end{equation}
	where the \(\sup\) is taken over all the \(p\) in \(Q\) for which the Alexandrov function $\alpha_1$ is defined (observe that $\alpha$ is defined for each \(t\geq 0\)).
	\begin{lemmaletter}[\cite{Korevaar1989TheSO}]\label{lemmaalphafunction1}
		Let \(M\) be a connected, properly embedded complete elliptic Weingarten surface in \(\R^3\), with or without boundary, and assume that there exists a solid vertical half cylinder \[\mathcal{C}\defeq\{x^2+y^2\leq K, z\geq0\},\] with \(K>0\), such that \(\emptyset\neq\partial(M\cap \mathcal{C})\subset\{z=0\}\) and \(M\cap C\) is transverse to the plane \(\{z=0\}\). Then, given a vertical plane \(Q\) with direction $\nu$, either \(M\) has a plane of symmetry in the direction $\nu$ or the function $\alpha$, defined with respect to \(M\cap C\), is strictly decreasing in \(t\).
	\end{lemmaletter}
	
	Lemma \ref{lemmaalphafunction1} basically says the following. Assume that \(M\) doesn't have a plane of symmetry parallel to \(Q\), and consider the horizontal sections \((M\cap\mathcal{C})\cap\{z=t\}\), with\,\,\(t\geq 0\). We can start the Alexandrov procedure for \(M\cap\mathcal{C}\) only focusing on what happens at a section\,\,\((M\cap\mathcal{C})\cap\{z=t\}\). Then, the Alexandrov procedure for the section \((M\cap\mathcal{C})\cap\{z=t\}\) stops \textit{strictly before} it stops for \((M\cap\mathcal{C})\cap\{z=\tilde{t}\}\), for each $\tilde{t}>t$.
	\section{A result on symmetry}
	 We use the following notation. $\mathcal{H}$ denotes the plane\,\(\{z=0\}\), and \(\R^3_+\) the upper halfspace determined by \(\mathcal{H}\). We usually deal with a surface \(M\) with boundary a simple closed curve in $\mathcal{H}$. We use \(C\) to denote that curve. Now, \(C\) separates the plane in two open components. We refer to the bounded one as the \textit{interior} and to the unbounded one as the \textit{exterior}. The interior of \(C\) will be denoted by \(D\).
	\begin{theorem}\label{zero}
	Let \(M\) be a properly embedded complete elliptic Weingarten surface in \(\R^3_+\). Assume\,\(M\) is transverse to the plane \(\mathcal{H}\) along \(C\) and that, outside from a compact set, \(M\) is contained in a finite number of solid half cylinders. Then, if \(P\) is a plane of Alexandrov symmetry for \(C\), whose direction is orthogonal to the axes of the cylinders, it holds that \(P\) is also a plane of symmetry for \(M\). In particular, if \(C\) is strictly convex and \(M\) is contained in a solid vertical cylinder, then \(M\) inherits all the symmetries of \(C\).
\end{theorem}
\begin{proof}
	Let $\mathcal{C}_1$,..., $\mathcal{C}_n$ be the solid half cylinders containing \(M\) outside from a compact set, and define $\mathcal{D}_1$,.., $\mathcal{D}_n$ to be the bases of the cylinders. One can assume that \(M\) meets the boundaries of the \(\mathcal{C}_i\)'s only on the $\mathcal{D}_i$'s, transversely. Then, Lemma \ref{lemmaalphafunction1} applies to each of the\,\,\(M\cap\mathcal{C}_i\)'s.
	
	We may assume the \(\mathcal{C}_i\)'s to be pairwise disjoint. Indeed, if \(\mathcal{C}_i\cap \mathcal{C}_j\neq\emptyset\), \(i\neq j\), and the intersection is bounded, we can just slide \(\mathcal{C}_i\) along its axis, until it doesn't cross \(\mathcal{C}_j\) anymore. On the other hand, when \(\mathcal{C}_i\cap \mathcal{C}_j\) is unbounded, we have that \(\mathcal{C}_i\) and \(\mathcal{C}_j\) must have parallel axes (and they must diverge in the same direction). Then, there exists another solid half cylinder $\mathcal{C}$ which contains both \(\mathcal{C}_i\) and $\mathcal{C}_j$, and we can replace them both with $\mathcal{C}$.
	
	Denote by $\sigma_1$,..., \(\sigma_n\) the boundaries of the \(M\cap\mathcal{C}_1\),..., \(M\cap\mathcal{C}_n\). Observe that each\,\,$\sigma_{i}$ may be not connected.
	\begin{figure}[!htb]
		\centering

		\tikzset{every picture/.style={line width=0.75pt}} 
		
		\begin{tikzpicture}[x=0.75pt,y=0.75pt,yscale=-0.8,xscale=0.8]
			
			\draw   (365.22,211.71) .. controls (365.22,211.81) and (365.23,211.92) .. (365.23,212.02) .. controls (365.23,219.19) and (347.36,225) .. (325.33,225) .. controls (303.29,225) and (285.42,219.19) .. (285.42,212.02) .. controls (285.42,211.94) and (285.43,211.85) .. (285.43,211.77) ;  
			\draw  [dash pattern={on 4.5pt off 4.5pt}] (285.43,211.77) .. controls (285.84,204.72) and (303.55,199.05) .. (325.33,199.05) .. controls (347.09,199.05) and (364.79,204.71) .. (365.22,211.76) ;  
			\draw    (365.22,211.76) .. controls (368.62,197.87) and (384,194.09) .. (393.13,183.87) .. controls (402.25,173.64) and (415.84,148.18) .. (434,141.04) ;
			\draw   (378.22,126.33) .. controls (377.78,122.24) and (387.16,116.7) .. (390.79,117.84) .. controls (394.41,118.97) and (396.2,123.8) .. (407.78,126.97) .. controls (419.36,130.13) and (407.84,140.36) .. (396.26,137.2) .. controls (384.68,134.04) and (378.66,130.41) .. (378.22,126.33) -- cycle ;
			\draw   (420.8,120.18) .. controls (423.85,119.67) and (438.3,130.54) .. (436.37,137.84) .. controls (434.45,145.14) and (425.6,138.68) .. (421.85,131.58) .. controls (418.1,124.49) and (417.75,120.69) .. (420.8,120.18) -- cycle ;
			\draw    (407.78,126.97) .. controls (413.89,124.82) and (414.54,120.93) .. (420.8,120.18) ;
			\draw   (386.01,106.04) .. controls (386.69,103.47) and (399.13,108.18) .. (400.35,111.11) .. controls (401.57,114.05) and (401.22,115.63) .. (397.82,116.01) .. controls (394.42,116.39) and (387.29,111.2) .. (386.73,110.28) .. controls (386.18,109.37) and (385.34,108.6) .. (386.01,106.04) -- cycle ;
			\draw    (363.93,146.23) .. controls (370.04,144.09) and (372.37,127.42) .. (379.5,122.88) ;
			\draw    (390.53,154.02) .. controls (396.64,151.88) and (410,141.04) .. (412.59,131.31) ;
			\draw    (236.76,112.5) .. controls (251.69,132.61) and (263.68,155.09) .. (279.59,157.26) .. controls (295.49,159.44) and (299,154.4) .. (310,166.4) .. controls (321,178.4) and (334.76,146.91) .. (341.89,154.7) .. controls (349.03,162.48) and (360.71,147.56) .. (365.25,133.94) .. controls (369.79,120.31) and (384.39,107.07) .. (386.01,106.04) ;
			\draw   (203.41,143.33) .. controls (198.55,138.06) and (202.07,126.9) .. (211.28,118.38) .. controls (220.49,109.87) and (231.9,107.23) .. (236.76,112.5) .. controls (241.63,117.76) and (238.11,128.93) .. (228.9,137.44) .. controls (219.69,145.95) and (208.28,148.59) .. (203.41,143.33) -- cycle ;
			\draw    (285.43,211.77) .. controls (276.34,178.67) and (262.72,205.28) .. (241.31,192.3) .. controls (219.89,179.32) and (215.35,155.32) .. (203.41,143.33) ;
			\draw    (394.42,120.28) .. controls (395.07,120.28) and (395.07,117.69) .. (397.82,116.01) ;
			\draw  [dash pattern={on 4.5pt off 4.5pt}]  (224,175.6) .. controls (224.75,165.1) and (225.92,154.1) .. (219.8,143.72) ;
			\draw   (211,133.6) .. controls (208.37,130.75) and (210.33,124.65) .. (215.39,119.97) .. controls (220.45,115.3) and (226.68,113.82) .. (229.32,116.67) .. controls (231.95,119.52) and (229.99,125.62) .. (224.93,130.29) .. controls (219.87,134.97) and (213.63,136.45) .. (211,133.6) -- cycle ;
			\draw    (219.8,143.72) .. controls (217.77,140.26) and (214.93,136.88) .. (211,133.6) ;
			\draw    (237.28,127.24) .. controls (234.95,123.05) and (232.27,119.44) .. (229.32,116.67) ;
			\draw  [dash pattern={on 4.5pt off 4.5pt}]  (241.31,192.3) .. controls (250.88,174.12) and (247.34,145.32) .. (237.28,127.24) ;
			\draw    (224,175.6) .. controls (233,176.4) and (240,185.4) .. (241.31,192.3) ;
			\draw  [color={rgb, 255:red, 208; green, 2; blue, 27 }  ,draw opacity=1 ] (191.58,154.71) .. controls (184.94,147.52) and (192.17,130.04) .. (207.74,115.68) .. controls (223.3,101.32) and (241.3,95.51) .. (247.94,102.7) .. controls (254.58,109.9) and (247.34,127.38) .. (231.78,141.74) .. controls (216.22,156.1) and (198.22,161.91) .. (191.58,154.71) -- cycle ;
			\draw [color={rgb, 255:red, 208; green, 2; blue, 27 }  ,draw opacity=1 ] [dash pattern={on 4.5pt off 4.5pt}]  (115,71.72) -- (134.75,93.12) ;
			\draw [color={rgb, 255:red, 208; green, 2; blue, 27 }  ,draw opacity=1 ]   (134.75,93.12) -- (191.58,154.71) ;
			\draw  [color={rgb, 255:red, 208; green, 2; blue, 27 }  ,draw opacity=1 ] (443.96,156.01) .. controls (451.19,146.43) and (440.26,125.99) .. (419.55,110.35) .. controls (398.84,94.71) and (376.19,89.79) .. (368.96,99.37) .. controls (361.73,108.94) and (372.66,129.38) .. (393.37,145.02) .. controls (414.08,160.66) and (436.73,165.58) .. (443.96,156.01) -- cycle ;
			\draw [color={rgb, 255:red, 208; green, 2; blue, 27 }  ,draw opacity=1 ] [dash pattern={on 4.5pt off 4.5pt}]  (512.99,64.6) -- (495.19,88.17) ;
			\draw [color={rgb, 255:red, 208; green, 2; blue, 27 }  ,draw opacity=1 ]   (495.19,88.17) -- (443.96,156.01) ;
			\draw [color={rgb, 255:red, 208; green, 2; blue, 27 }  ,draw opacity=1 ] [dash pattern={on 4.5pt off 4.5pt}]  (171.36,19.72) -- (191.11,41.12) ;
			\draw [color={rgb, 255:red, 208; green, 2; blue, 27 }  ,draw opacity=1 ]   (191.11,41.12) -- (247.94,102.7) ;
			\draw [color={rgb, 255:red, 208; green, 2; blue, 27 }  ,draw opacity=1 ] [dash pattern={on 4.5pt off 4.5pt}]  (437.99,7.96) -- (420.19,31.53) ;
			\draw [color={rgb, 255:red, 208; green, 2; blue, 27 }  ,draw opacity=1 ]   (420.19,31.53) -- (368.96,99.37) ;
			
			\draw (268,211.4) node [anchor=north west][inner sep=0.75pt]  [xscale=1,yscale=1]  {$C$};
			\draw (307,139.4) node [anchor=north west][inner sep=0.75pt]  [xscale=1,yscale=1]  {$\tilde{M}$};
			\draw (215,48.4) node [anchor=north west][inner sep=0.75pt]  [xscale=1,yscale=1]  {$\mathcal{C}_{1}$};
			\draw (374,47.4) node [anchor=north west][inner sep=0.75pt]  [xscale=1,yscale=1]  {$\mathcal{C}_{2}$};
			\draw (192,98.4) node [anchor=north west][inner sep=0.75pt]  [xscale=1,yscale=1]  {$\sigma _{1}$};
			\draw (414,92.4) node [anchor=north west][inner sep=0.75pt]  [xscale=1,yscale=1]  {$\sigma _{2}$};

		\end{tikzpicture}
		\caption{The compact surface with boundary \(\tilde{M}\).}
		\label{figuretildem}
	\end{figure}
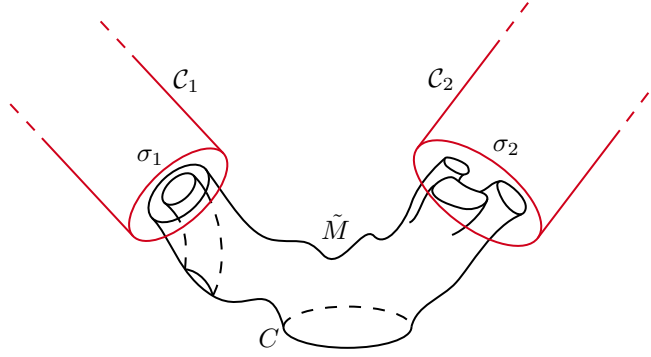
	
	Let \(P\) be a plane of Alexandrov symmetry for \(C\) whose direction $\nu$ is orthogonal to the axes of the cylinders \(\mathcal{C}_1\),...,\(\mathcal{C}_n\).
	
	Assume by contradiction that \(M\) is not symmetric with respect to \(P\). First observe that, if one among the \(M\cap\mathcal{C}_i\)'s was symmetric in the direction $\nu$, then this would imply a symmetry for a region of \(M\) containing \(C\). Since we assumed that \(M\) is not symmetric with respect to \(P\), this would imply that \(C\subsetneq M\cap\mathcal{H}\), which is absurd by the definition of \(M\). Therefore, by Lemma\,\,\ref{lemmaalphafunction1}, given a plane \(Q\) with direction $\nu$ and disjoint from \(M\), we can assume that, for every \(M\cap\mathcal{C}_i\), the corresponding function $\alpha$ is strictly decreasing.
	
	We can apply the Alexandrov reflection technique to \(M\), starting from the plane\,\,\(Q\), only looking at what happens in the bounded surface \(\tilde{M}\defeq M\setminus\bigcup_i(M\cap\mathcal{C}_i)\) (see Figure \ref{figuretildem}). Observe that the procedure stops for either:
	\begin{itemize}
		\item[(i)] An interior touching point for $\tilde{M}$;
		\item[(ii)] A touching point in some $\sigma_i$;
		\item[(iii)] A touching point in \(C\).
	\end{itemize}
	
	Now, in the cases (i) and (ii) the monotonicity of the Alexandrov functions implies that an interior touching point for \(M\) must have occurred. Thus, by an application of the interior maximum principle we would get a contradiction since we would have found a plane of symmetry in the direction $\nu$.
	
	We can then assume that the procedure stops, at the plane \(P\equiv Q(0)\), for a touching point in \(C\). By definition of Alexandrov symmetry, this means that the reflection\,\,\(C^*(0)\) coincides with \(C\cap\Pi_+(0)\). In general, this doesn't imply that the reflected $\tilde{M}^*(0)$ is tangent to $\tilde{M}$ along \(C\). Nevertheless, we can repeat the same argument in the direction \(-\nu\). By definition of Alexandrov symmetry for \(C\), the procedure stops at the same plane \(P\). This forces \(\tilde{M}^*(0)\) and $\tilde{M}$ to be tangent along \(C\). Application of the maximum principle at the boundary then implies that \(P\) is a plane of symmetry for \(M\).
\end{proof}
		The preceding theorem is sharp in the following sense. Consider a tilted cylinder having its boundary on an ellipse \(C\). Then, the surface shares only one of the two symmetries of its boundary, and it is the one detected by Theorem \ref{zero}. Observe, in addition, that the condition of being cylindrically bounded is quite natural. In fact, it is shown in \cite{saearpweingarten}, \cite{espinar} that every annular end of an elliptic Weingarten surface is cylindrically bounded when \(f(0)\neq0\):
	\begin{theoremletter}[\cite{saearpweingarten}, \cite{espinar}]\label{saearp1}
		Let \(A\approx\esse^1\times[0,1)\) be a properly embedded annulus in \(\R^3\), satisfying \(H=f(H^2-K)\), with \(f\) elliptic and such that \(f(0)\neq0\). Then, \(A\) is cylindrically bounded.
	\end{theoremletter}
	We remark that Theorem \ref{saearp1} was first proved in \cite{saearpweingarten}, under the additional assumption that the Weingarten class satisfies height estimates. Those were proven, in the same article, under some additional hypotheses for the function \(f\).	Instead, in \cite{espinar}, the height estimates are proved by a clever application of the Alexandrov reflection principle, only requiring, besides the ellipticity, the Weingarten class to admit a sphere -- i.e. \(f(0)\neq0\).
	\section{The case of linear Weingarten surfaces}\label{sectionthesecondtheorem}
	
	 		In the following, we have to restrict ourselves to surfaces satisfying \(2aH+bK=1\), with \(a\), \(b\) positive constants. Note that, for the ellipticity, we would only need \(a^2+b>0\). Nevertheless, we ask for \(a,b>0\) because of some stronger results that we specify in the following. 
	 		
	 		Elliptic linear Weingarten surfaces satisfy the assumptions of Theorem \ref{saearp1}, and so every annular end of an elliptic linear Weingarten surface is contained in a solid half cylinder. This implies, first, that it makes sense to speak of a \textit{vertical} end or, more generally, of an end with a given direction: the direction of the end will be the axis of a solid cylinder which contains the end -- in addition, we call a vertical end \textit{positive} if it diverges in the upper halfspace and \textit{negative} if it diverges in the lower halfspace. Moreover, we strongly rely on the fact that the annular ends of a properly embedded, complete, linear Weingarten surface with \(a,b>0\) converge geometrically to a rotational Weingarten surface -- which is a part of the theory of Korevaar, Kusner and Solomon which has been developed, in the Weingarten setting, only for the linear Weingarten's with \(a,b>0\) \cite{saearpweingarten}:
	 		\begin{definition}
	 			Let \(A\) be a properly embedded, cylindrically bounded annulus in \(\R^3\), with axis vector $\nu$. Assume that \(A\) diverges only in the direction where $\nu$ is pointing. Given a diverging, non-decreasing sequence \(\{t_k\}_{k\in\N}\subset\R^+\), define \(A_k\defeq A-t_k\nu\). The family \(\{A_k\}_{k\in\N}\) will be called a \textit{slide-back sequence} for \(A\).
	 		\end{definition}
	 		\begin{theoremletter}[\cite{saearpweingarten},\cite{Korevaar1989TheSO}]\label{geometricconvergence}
	 			Let \(A\) be a properly embedded annulus in \(\R^3\) satisfying \(2aH+bK=1\), with \(a,b>0\), and \(\{A_k\}_{k\in\N}\) a slide-back sequence of \(A\). Then, there exists a rotational linear Weingarten surface \(W\), satisfying \(2aH+bK=1\) and sharing the same axis as \(A\), such that, for each compact subset \(K\subset W\), a subsequence of\,\(\{A_k\}_{k\in\N}\) converges, as normal graphs, to \(0\), in the \(C^k\) topology for each \(k\in\N\).
	 		\end{theoremletter}
	 			 		
	 		The embedded rotational linear Weingarten surfaces have been described in \cite{saearpweingarten}, and, for \(a>0\) and \(b>0\) fixed in the linear Weingarten equation, they behave just like the classical Delaunay surfaces \cite{Delaunay1841}, \cite{Korevaar1989TheSO}: they form a one parameter family of rotational, periodic surfaces, ranging between a cylinder and a stack of \textit{American footballs} \cite{fernandezweingarten} \cite{saearpweingarten} -- the latter failing to be embedded. Except for the cylinder, they have a series of alternating necks and bulges, where the distance from the axis of rotation -- the radius -- attains, respectively, its minimum and maximum (we will speak of \textit{small} and \textit{big} radius) and the normal to the surface is orthogonal to the axis of rotation. From now on, we will call\textit{W-Delaunay}'s the rotational linear Weingarten surfaces with \(a,b>0\). We remark that W-Delaunay's are parallel surfaces to classical Delaunay surfaces \cite{saearpweingarten}.
	 		
 	 		We don't know if the convergence of the annular ends to a rotational Weingarten surface can be proved for a more general Weingarten equation (see \textit{Question 2} in \cite{saearpfrench}) -- not even in the linear case, assuming only \(a^2+b>0\), i.e. the condition for ellipticity. The question arises naturally, since many classes of Weingarten surfaces admit Delaunay-like surfaces, under some hypothesis on \(f\) \cite{saearpfrench}. To follow the proof of \cite{Korevaar1989TheSO}, what is missing is, mainly, a bound on the curvature. Then, one would use a compactness argument for a sequence of functions satisfying the Weingarten equation. This last issue is addressed in \cite{fernandezgalvezmira} for a uniformly elliptic equation.
	 		
	 		Besides the convergence, what one would like to use is flux formula, which is a classical tool in the cmc theory. Just like the convergence, flux formula has been described in \cite{saearpweingarten} only for the linear Weingarten's with \(a,b>0\), and in \cite{someproblems} it is described as an open problem to decide which classes of Weingarten surfaces admit a ``good''  flux\,\,formula.
	 		
	 		Observe that, on a linear Weingarten surface with \(a,b>0\), the mean curvature is nowhere vanishing. Indeed, at a point \(p\) for which \(H(p)=0\), the Gauss curvature \(K\) should verify \(bK(p)=1\), with \(b>0\), which is absurd since we know that \(K(p)\leq0\).
	 		
	 		In the following lemma we compute what, in the language of \cite{Korevaar1989TheSO}, is called the \textit{mass} of the end of a W-Delaunay surface. The concept of mass of an end is strictly related to flux formula and, in fact, we use it in the proof of Theorem \ref{two} to evaluate the contribution of the ends to the flux.

\begin{lemma}\label{lemmamassadelaunay}
	Consider a vertical linear W-Delaunay \(M\), oriented according to its mean curvature vector. Let \(K\) be a compact, connected, orientable surface having its boundary on a simple loop $\Gamma$ in \(M\), being non-contractible in \(M\). Now, \(\Gamma\) separates \(M\) in two annular ends: let \(A\) be the positive one -- that is, the one diverging in the upper halfspace. Let $\nu_\Gamma$ be the conormal to $\Gamma$ pointing toward \(A\) and let \(n_K\) be the normal vector field to \(K\) making \(\bar{K}\cup A\) an oriented surface (with corners). Then, it holds that:
\begin{equation}\label{14gennaio}
		\int_K\pe{e_3,n_K}-\frac{1}{2}\int_\Gamma\pe{e_3,(2a+bT)\nu_\Gamma}=\pi(Rr+b).
\end{equation}
where \(R\) and \(r\) are the big and small radii of \(M\).
\end{lemma}
\begin{proof}
	First observe that the left-hand side of \eqref{14gennaio} is, in a sense, homotopically invariant. Indeed, consider a smooth simple loop $\tilde{\Gamma}$ on \(M\), homotopic to \(\Gamma\) and disjoint from $\Gamma$, and let \(\tilde{K}\) be a compact, connected, orientable surface having its boundary on $\tilde{\Gamma}$. Now, $\Gamma$ and $\tilde{\Gamma}$ bound a connected surface \(B\), homeomorphic to an annulus. We apply flux formula \cite{saearpweingarten} to the oriented compact cycle obtained by joining \(K\) and $\tilde{K}$ to \(B\), and choosing the right orientation on the formers. One has:
	\begin{equation}
		\int_K\pe{e_3,n_K}-\frac{1}{2}\int_\Gamma\pe{e_3,(2a+bT)\nu_\Gamma}=\int_{\tilde{K}}\pe{e_3,n_{\tilde{K}}}-\frac{1}{2}\int_{\tilde{\Gamma}}\pe{e_3,(2a+bT)\nu_{\tilde{\Gamma}}}
	\end{equation}
	where \(n_{\tilde{K}}\) and \(\nu_{\tilde{\Gamma}}\) are chosen according to \(n_K\) and \(\nu_\Gamma\): the annular end $\tilde{A}$, as in the statement, is the positive annular end bounded by $\tilde{\Gamma}$ on \(M\).
	
	Thus, we can compute\,\(\int_{K}\pe{e_3,n_{K}}-\frac{1}{2}\int_{\Gamma}\pe{e_3,(2a+bT)\nu_\Gamma}\) assuming \(\Gamma\) is a parallel with vertical co-normal \(\nu\) and \(K\) a planar disc. We fix $\Gamma$ to be the parallel at a bulge, and \(K\) the planar disc with boundary $\Gamma$. We start by computing \((2a+bT)\nu_\Gamma\). Recall that\,\(T=2HI-A\), \(A\) being the shape operator of the surface. Now, on a surface of revolution, the principal curvatures are \(-\frac{d\psi}{ds}\) and \(\frac{\cos\psi}{y}\), where $\psi$ is the angle that the generating curve makes with the axis of rotation and \(y\) is the radius at the point. Moreover, if we consider the standard parametrization of a surface of revolution, we see that \(A\) is diagonal, and $\nu_\Gamma$ is a multiple of a coordinate vector field, so that it is an eigenvector of \(A\). Thus:
	\[\left(2a+bT\right)\nu_\Gamma=\left(2a+bR^{-1}\right)\nu_\Gamma\]
	and
	\begin{equation}\label{wdel}
		\int_{K}\pe{e_3,n_{K}}-\frac{1}{2}\int_{\Gamma}\pe{e_3,(2a+bT)\nu_\Gamma}=-\pi\left(R^2-2aR-b\right).
	\end{equation}
	We remark that, by the homotopical invariance, we could have computed last integral also on a neck. Thus, if \(r\) is the radius of a neck:
	\[R^2-2aR-b=r^2-2ar-b,\]
	so that we get the following relation between the two radii: \begin{equation}\label{relation}
		R+r=2a.
	\end{equation}
	This holds, in principle, only when the two radii are different. But, since the W-Delaunay's form a one parameter family, the same relation holds, by continuity, also for the cylinder (\(R=r\)). This concludes the proof.
\end{proof}

 Next lemma is used in the proof of Theorem \ref{two} to estimate the number of intersections of a surface with a plane, under some topological conditions.

\begin{lemma}\label{lemmacountingloops}
	Let \(r\) be a properly embedded curve in \(\R^3_+\), homeomorphic to \([0,1)\), having its endpoint \(p\) on the plane \(\mathcal{H}\) and being transverse to $\mathcal{H}$ in \(p\). Consider a compact connected orientable surface \(S\) in \(\R^3_+\setminus r\) having its boundary on the plane \(\mathcal{H}\) and being transverse to the plane. Then, the number of loops in $\partial S$ being non-contractible in $\mathcal{H}\setminus\{p\}$ is either even or zero.
\end{lemma}
\begin{proof}
	Assume that there exists a simple loop $\gamma$ in $\partial S$, being non-contractible in \(\mathcal{H}\setminus\{p\}\). Let \(\tilde{r}\) be the proper curve obtained by joining to \(r\) the vertical, negative half line orthogonal to $\mathcal{H}$ and with endpoint \(p\). We may assume that \(r\) joins smoothly with the half line.
	
	Observe that, given an embedded closed surface \(\Sigma\) in \(\R^3\), intersecting $\tilde{r}$ transversely in a finite number of points, it must hold that the cardinality of $\tilde{r}\cap\Sigma$ is even. We now show that, if the cardinality of the non-contractible loops in $\partial S$ was odd, then one would be able to construct a certain embedded closed surface, meeting $\tilde{r}$ transversely, with an odd number of intersections with the curve.
		\begin{figure}[!htb]

		\tikzset{every picture/.style={line width=0.6pt}} 
		
		\begin{tikzpicture}[x=0.75pt,y=0.75pt,yscale=-0.8,xscale=0.8]
			
			\draw    (191.37,215) -- (466.02,215) ;
			\draw    (191.37,215) -- (247.23,153.65) ;
			\draw   (229.48,188.7) .. controls (229.48,186.04) and (237.49,183.88) .. (247.38,183.88) .. controls (257.26,183.88) and (265.28,186.04) .. (265.28,188.7) .. controls (265.28,191.36) and (257.26,193.51) .. (247.38,193.51) .. controls (237.49,193.51) and (229.48,191.36) .. (229.48,188.7) -- cycle ;
			\draw   (236.48,188.7) .. controls (236.48,187.4) and (240.34,186.35) .. (245.12,186.35) .. controls (249.89,186.35) and (253.76,187.4) .. (253.76,188.7) .. controls (253.76,189.99) and (249.89,191.04) .. (245.12,191.04) .. controls (240.34,191.04) and (236.48,189.99) .. (236.48,188.7) -- cycle ;
			\draw [color={rgb, 255:red, 208; green, 2; blue, 27 }  ,draw opacity=1 ]   (229.48,188.7) -- (229.48,201.99) ;
			\draw [color={rgb, 255:red, 208; green, 2; blue, 27 }  ,draw opacity=1 ]   (265.28,188.7) -- (265.28,201.99) ;
			\draw  [color={rgb, 255:red, 208; green, 2; blue, 27 }  ,draw opacity=1 ] (264.96,201.07) .. controls (265.17,201.37) and (265.28,201.67) .. (265.28,201.99) .. controls (265.28,204.65) and (257.26,206.8) .. (247.38,206.8) .. controls (237.49,206.8) and (229.48,204.65) .. (229.48,201.99) .. controls (229.48,201.39) and (229.88,200.82) .. (230.62,200.29) ;  
			\draw  [color={rgb, 255:red, 208; green, 2; blue, 27 }  ,draw opacity=1 ][dash pattern={on 4.5pt off 4.5pt}] (229.58,202.5) .. controls (229.52,202.33) and (229.48,202.16) .. (229.48,201.99) .. controls (229.48,199.33) and (237.49,197.17) .. (247.38,197.17) .. controls (257.26,197.17) and (265.28,199.33) .. (265.28,201.99) .. controls (265.28,202.1) and (265.26,202.21) .. (265.23,202.33) ;  
			\draw  [color={rgb, 255:red, 208; green, 2; blue, 27 }  ,draw opacity=1 ] (253.76,196.51) .. controls (253.79,196.6) and (253.81,196.69) .. (253.81,196.79) .. controls (253.81,198.08) and (249.95,199.13) .. (245.17,199.13) .. controls (240.6,199.13) and (236.86,198.17) .. (236.55,196.95) ;  
			\draw [color={rgb, 255:red, 208; green, 2; blue, 27 }  ,draw opacity=1 ]   (236.48,188.7) -- (236.48,196.51) ;
			\draw [color={rgb, 255:red, 208; green, 2; blue, 27 }  ,draw opacity=1 ]   (253.76,188.7) -- (253.76,196.51) ;
			\draw  [color={rgb, 255:red, 208; green, 2; blue, 27 }  ,draw opacity=1 ][dash pattern={on 4.5pt off 4.5pt}] (237.22,197.7) .. controls (236.78,197.42) and (236.53,197.11) .. (236.53,196.79) .. controls (236.53,195.49) and (240.4,194.44) .. (245.17,194.44) .. controls (249.95,194.44) and (253.81,195.49) .. (253.81,196.79) .. controls (253.81,196.83) and (253.81,196.87) .. (253.8,196.9) ;  
			\draw    (265.28,188.7) .. controls (264.87,181) and (269.39,183.47) .. (270.21,181.41) ;
			\draw    (229.48,188.7) .. controls (225.37,182.24) and (227.42,176.48) .. (223.31,175.24) ;
			\draw    (236.48,188.7) .. controls (236.48,183.47) and (231.13,182.24) .. (230.72,177.3) ;
			\draw    (253.76,188.7) .. controls (258.28,185.12) and (250.05,181.83) .. (257.87,177.71) ;
			\draw    (342.14,173.31) .. controls (357.29,155.21) and (346.85,135.42) .. (340.65,117.2) ;
			\draw  [dash pattern={on 4.5pt off 4.5pt}]  (340.65,117.2) .. controls (335.84,103.07) and (333.59,89.89) .. (347.85,79.2) ;
			\node[draw, circle, inner sep=0.5pt, fill] at (341.69,173.31) {};
			
			\draw   (316.41,164.3) .. controls (342.79,160.62) and (394.64,151.7) .. (398.84,164.82) .. controls (403.04,177.95) and (412.49,174.27) .. (418.14,181.17) .. controls (423.79,188.06) and (325.82,192.21) .. (299.44,181.17) .. controls (273.06,170.13) and (290.03,167.98) .. (316.41,164.3) -- cycle ;
			\draw    (286.06,172.32) .. controls (286.59,152.01) and (282.13,142.34) .. (268.72,138.61) ;
			\draw    (418.14,181.17) .. controls (417.6,157.73) and (422.13,146.55) .. (435.71,142.25) ;
			\draw   (428.27,181.25) .. controls (428.27,178.59) and (436.28,176.44) .. (446.16,176.44) .. controls (456.05,176.44) and (464.06,178.59) .. (464.06,181.25) .. controls (464.06,183.91) and (456.05,186.07) .. (446.16,186.07) .. controls (436.28,186.07) and (428.27,183.91) .. (428.27,181.25) -- cycle ;
			\draw   (435.26,181.25) .. controls (435.26,179.96) and (439.13,178.91) .. (443.9,178.91) .. controls (448.67,178.91) and (452.54,179.96) .. (452.54,181.25) .. controls (452.54,182.55) and (448.67,183.6) .. (443.9,183.6) .. controls (439.13,183.6) and (435.26,182.55) .. (435.26,181.25) -- cycle ;
			\draw [color={rgb, 255:red, 208; green, 2; blue, 27 }  ,draw opacity=1 ]   (428.27,181.25) -- (428.27,194.54) ;
			\draw [color={rgb, 255:red, 208; green, 2; blue, 27 }  ,draw opacity=1 ]   (464.06,181.25) -- (464.06,194.54) ;
			\draw  [color={rgb, 255:red, 208; green, 2; blue, 27 }  ,draw opacity=1 ] (463.74,193.63) .. controls (463.95,193.92) and (464.06,194.23) .. (464.06,194.54) .. controls (464.06,197.2) and (456.05,199.36) .. (446.16,199.36) .. controls (436.28,199.36) and (428.27,197.2) .. (428.27,194.54) .. controls (428.27,193.94) and (428.67,193.37) .. (429.41,192.85) ;  
			\draw  [color={rgb, 255:red, 208; green, 2; blue, 27 }  ,draw opacity=1 ][dash pattern={on 4.5pt off 4.5pt}] (428.37,195.06) .. controls (428.3,194.89) and (428.27,194.72) .. (428.27,194.54) .. controls (428.27,191.88) and (436.28,189.73) .. (446.16,189.73) .. controls (456.05,189.73) and (464.06,191.88) .. (464.06,194.54) .. controls (464.06,194.66) and (464.05,194.77) .. (464.02,194.88) ;  
			\draw  [color={rgb, 255:red, 208; green, 2; blue, 27 }  ,draw opacity=1 ] (452.54,189.07) .. controls (452.58,189.16) and (452.6,189.25) .. (452.6,189.34) .. controls (452.6,190.64) and (448.73,191.69) .. (443.96,191.69) .. controls (439.39,191.69) and (435.64,190.72) .. (435.34,189.5) ;  
			\draw [color={rgb, 255:red, 208; green, 2; blue, 27 }  ,draw opacity=1 ]   (435.26,181.25) -- (435.26,189.07) ;
			\draw [color={rgb, 255:red, 208; green, 2; blue, 27 }  ,draw opacity=1 ]   (452.54,181.25) -- (452.54,189.07) ;
			\draw  [color={rgb, 255:red, 208; green, 2; blue, 27 }  ,draw opacity=1 ][dash pattern={on 4.5pt off 4.5pt}] (436,190.26) .. controls (435.56,189.98) and (435.32,189.67) .. (435.32,189.34) .. controls (435.32,188.05) and (439.19,187) .. (443.96,187) .. controls (448.73,187) and (452.6,188.05) .. (452.6,189.34) .. controls (452.6,189.38) and (452.6,189.42) .. (452.59,189.46) ;  
			\draw    (464.06,181.25) .. controls (463.65,173.56) and (468.18,176.03) .. (469,173.97) ;
			\draw    (428.27,181.25) .. controls (424.15,174.79) and (426.21,169.03) .. (422.1,167.8) ;
			\draw    (435.26,181.25) .. controls (435.26,176.03) and (429.91,174.79) .. (429.5,169.85) ;
			\draw    (452.54,181.25) .. controls (457.07,177.67) and (448.84,174.38) .. (456.66,170.27) ;
			
			\draw (188.84,185.13) node [anchor=north west][inner sep=0.75pt]  [xscale=0.8,yscale=0.8]  {$\mathcal{H}$};
			\draw (329.73,164.96) node [anchor=north west][inner sep=0.75pt]  [xscale=0.8,yscale=0.8]  {$p$};
			\draw (292.51,184) node [anchor=north west][inner sep=0.75pt]  [xscale=0.8,yscale=0.8]  {$\gamma $};
			\draw (343.88,105.4) node [anchor=north west][inner sep=0.75pt]  [xscale=0.8,yscale=0.8]  {$r$};
			\draw (264.96,144) node [anchor=north west][inner sep=0.75pt]  [xscale=0.8,yscale=0.8]  {$S$};
			\draw (236.64,119.55) node [anchor=north west][inner sep=0.75pt]  [xscale=1,yscale=1]  {$...$};
			\draw (446.6,144.86) node [anchor=north west][inner sep=0.75pt]  [xscale=1,yscale=1]  {$...$};

		\end{tikzpicture}
		\caption{Constructing a closed surface starting from \(S\).}
		\label{figureclosingsurface}
	\end{figure}
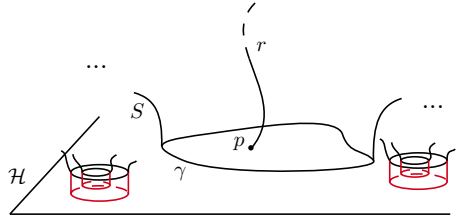
	
	Now, assume first by contradiction that \(\gamma\) is the only non-contractible loop in $\partial S$. Then, one can form a closed surface intersecting $\tilde{r}$ in exactly one point by attaching to \(S\), along every loop in $\partial S$, the interior of the loop in $\mathcal{H}$.

	When several loops are concentric, we shall attach to the outer loop a vertical cylinder in \(\R^3_-\), bounded by a planar cap at a certain negative height -- by doing this inductively, with all the loops contained in the interior of the original one, a choice of the negative heights is possible, so that the resulting surface doesn't have self-intersections (check Figure \ref{figureclosingsurface}).
	
	The resulting surface is a closed embedded connected surface $\Sigma$ with corners, still everything we said on the number of intersections with $\tilde{r}$ holds, since $\tilde{r}$ doesn't touch the corners of \(\Sigma\), and so it makes sense to speak of a transverse intersection. But then we get an absurd, since $\tilde{r}$ intersects $\Sigma$ in exactly one point -- that is, in the planar cap associated to $\gamma$.
	
	Then, there must be some other non-contractible loop. Assume by contradiction that the other non-contractible loops are $\gamma_1$,..., \(\gamma_{2k}\), with \(k\in\{1,2,...\}\). Up to changing the names, we can order them from the innermost to the outermost -- that is, $\gamma_{i}$ is contained in the interior of $\gamma_{i+1}$ for each \(i\). We treat the contractible loops as in the previous case. For the non-contractible loops, we attach the annular planar region between \(\gamma_{2l-1}\) and \(\gamma_{2l}\) to the two loops, for each \(l\) in \(\{1,...,k\}\). Regarding $\gamma$, we attach to it its interior on $\mathcal{H}$. As in the previous case, we get a contradiction. Indeed, the resulting surface is closed -- this comes from the fact that we assumed the non-contractible loops to be even -- and $\tilde{r}$ intersects it transversely in exactly one point -- i.e. in the interior of $\gamma$.
\end{proof}
In the proof of the main theorem, we use the theory of Alexandrov functions that we recalled in Section \ref{sectionpreliminaries}. The following lemma will be used to estimate when the Alexandrov procedure stops for a touching point on an end.
\begin{lemma}\label{lemmacircles}
	Let \(A\subset\R^3_+\) be a vertical annular end of a linear Weingarten surface, with \(a,b>0\). Given a vertical plane \(Q\) with normal $\nu$, define the functions $\alpha_1$ and $\alpha$ as in Section \ref{sectionpreliminaries}. Then, for each diverging sequence \(\{t_n\}_{n\in\N}\subset\R^+\), there exists a value \(T>0\) and  a subsequence \(\{t_{n_k}\}_{k\in\N}\) such that \begin{equation}\label{equationconvergencealpha}
		\alpha(t_{n_k}+T)\to d,
	\end{equation} where \(d\) is the distance between \(Q\) and the axis of \(A\).
\end{lemma}
\begin{proof}
	Consider a diverging sequence \(\{t_n\}_{n\in\N}\subset\R^+\). Let \(W\) be the W-Delaunay with the same axis as \(A\) given by Theorem\,\,\ref{geometricconvergence}. We know that \(W\) has a sequence of alternating necks and bulges. Consider \(T\in\R\) such that \(W\cap\{z=T\}\) is a neck -- or, equivalently, a bulge -- and let \(K\) be a compact subset of \(W\) containing \(W\cap\{z=T\}\). By Theorem\,\,\ref{geometricconvergence}, there exists a subsequence \(\{t_{n_k}\}_{k\in\N}\subset\{t_n\}_{n\in\N}\) such that the slide-back sequence \(\{A_{n_k}\}_{k\in\N}\) converges smoothly to \(K\) as normal graphs. In particular, \(\{A_{n_k}\cap\{z=T\}\}_{k\in\N}\) is a sequence of simple planar loops converging smoothly to the circle \(W\cap\{z=T\}\) as normal graphs. Then, by the definition of $\alpha_1$ and $\alpha$, the subsequence \(\{t_{n_k}\}_{k\in\N}\) verifies\,\,\eqref{equationconvergencealpha}.
\end{proof}
Before stating the main theorem, we introduce the following terminology to simplify the exposition. Given an annular end \(E\) satisfying a linear Weingarten equation \(2aH+bK=1\), with \(a,b>0\), we know by Theorem \ref{geometricconvergence} that \(E\) converges geometrically to a W-Delaunay surface \(W\), with big and small radii \(R\) and \(r\) respectively. In the following, we will refer as the \textit{mass} of \(E\) to the quantity \(\mathcal{W}\defeq\pi(Rr+b)\).
\begin{theorem}\label{two}
	Let \(M\) be a complete, properly embedded surface satisfying an equation\,\,\(2aH+bK=1\), for \(a\), \(b\) positive constants, with boundary a strictly convex planar curve\,\,\(C\subset\mathcal{H}\). Let \(D\) be the interior of \(C\). Assume that \(M\) is transverse to \(\mathcal{H}\) along\,\,\(C\) and \(M\) is contained, say, in the upper halfspace near \(C\). Assume that \(M\) has a finite number \(n\geq0\) of vertical annular ends. We index with \(i\in I\) the positive ends and with \(j\in J\) the negative ones.
	
	If \(M\) is compact, then \(M\) is contained in the upper halfspace.
	
	If \(M\) is not compact, then either \(M\) is contained in the upper halfspace or there is a simple loop $\gamma\subset M\cap(\mathcal{H}\setminus\bar{D})$, such that $\gamma$ generates \(\Pi_1(\mathcal{H}\setminus\bar{D})\). In the latter case, it holds:
	\begin{itemize}
		\item \(|D|\leq\sum_{i\in I}\mathcal{W}_i-\sum_{j\in J}\mathcal{W}_j\) when \(\sum_{i\in I}\mathcal{W}_i-\sum_{j\in J}\mathcal{W}_j\geq0\).
		\item \(|D|\geq\sum_{j\in J}\mathcal{W}_j-\sum_{i\in I}\mathcal{W}_i\) when \(\sum_{i\in I}\mathcal{W}_i-\sum_{j\in J}\mathcal{W}_j<0\);
	\end{itemize}
	where, for each \(i\in I\cup J\), $\mathcal{W}_i$ is the mass of the \(i\)-th end.
\end{theorem}
\begin{proof}
	We first tackle the case in which the following balance holds:
		\begin{equation}\label{equationbalance}
		\sum_{i\in I}\mathcal{W}_i-\sum_{j\in J}\mathcal{W}_j\geq0.
		\end{equation}
		The case of negative balance will be treated at the end, using some of the machinery introduced in the present case. Observe that the compact case is included in \eqref{equationbalance}.
		
	We follow the steps of the proof of \cite[Theorem 2]{saearpconvex}. Nevertheless, the presence of negative ends introduces certain technical difficulties, which do not permit to immediately translate what is done there.
	
	The compact case is explained in \cite{britosaearpmeeks}, \cite{saearpweingarten}. It can also be inferred from the following proof by simply neglecting the ends. Namely, one can ignore Step\,\,1 below, and pass directly to Step\,\,2, putting to zero the contribution of the ends to flux formula \eqref{balancing1}. Regarding the application of Alexandrov reflection principle (Step 3), one can apply it directly, with no need to use Lemma \ref{lemmaalphafunction1}. If \(M\) is not contained in \(\R^3_+\), the inequality in the thesis is then reached with a zero in the right-hand side (there are no ends), thereby leading to a contradiction.
	
	Let us now pass to the general case. Up to a small vertical translation, by Sard's Theorem, we may assume that \(M\) is transverse to the plane \(\mathcal{H}\). Let \(M_+\) be the component of \(M\cap\R^3_+\) containing \(C\).
	
	\textbf{Step 1:} \textit{If \(M\) is not contained in \(\R^3_+\) and \(M_+\cap(\mathcal{H}\setminus\bar{D})=\emptyset\), then \(M_+\) is compact.}
	\\Assume by contradiction that this is not the case. If \(M\) is not contained in\,\,\(\R^3_+\) and \(M_+\cap (\mathcal{H}\setminus\bar{D})=\emptyset\), it must hold that  and \(M\cap D\neq\emptyset\).
	
	This implies that \(\partial M_+\) is made up of \(C\) and a finite number of disjoint simple loops contained in \(D\). Now, we attach to \(M_+\) a proper subdomain \(D_0\) of \(D\), that is defined inductively as follows (refer to Figure \ref{figureconstructionD0}). Consider $\partial M_+\setminus C$. It is a finite union of disjoint simple loops in \(D\). Focus on the outer loops -- that is, the ones that may be connected to \(C\) via a path that doesn't cross other loops of \(\partial M_+\setminus C\). Define $B_1$ as the union of the interiors of the outer loops and \(A_1\) as:
	\[A_1\defeq D\setminus \bar{B}_1\]
	
	We now start the induction. Focus on the loops in $B_1$. We call outer loops the ones that may be connected to $\partial B_1$ via a path in $B_1$ that doesn't cross other loops. Define $\tilde{B}_2$ as the union of the interiors of those outer loops.
	
	Now, in $\tilde{B}_2$, define the outer loops analogously, and let \(B_2\) be the union of their interiors. Define:
	\[A_2\defeq \tilde{B}_2\setminus\bar{B}_2.\]
	
	To define \(A_k\), we do the same procedure starting from \(B_{k-1}\). When either $B_{k-1}$ or \(\tilde{B}_k\) don't contain any loop in their interior, we set \(\tilde{B}_k\) or $B_k$ equal to $\emptyset$ and stop the induction. Define then \(D_0\) as the (finite) union of the \(A_i\)'s.
	
	\begin{figure}[!htb]
		\centering

\tikzset{
	pattern size/.store in=\mcSize, 
	pattern size = 5pt,
	pattern thickness/.store in=\mcThickness, 
	pattern thickness = 0.3pt,
	pattern radius/.store in=\mcRadius, 
	pattern radius = 1pt}
\makeatletter
\pgfutil@ifundefined{pgf@pattern@name@_t43tg3toy}{
	\pgfdeclarepatternformonly[\mcThickness,\mcSize]{_t43tg3toy}
	{\pgfqpoint{0pt}{0pt}}
	{\pgfpoint{\mcSize+\mcThickness}{\mcSize+\mcThickness}}
	{\pgfpoint{\mcSize}{\mcSize}}
	{
		\pgfsetcolor{\tikz@pattern@color}
		\pgfsetlinewidth{\mcThickness}
		\pgfpathmoveto{\pgfqpoint{0pt}{0pt}}
		\pgfpathlineto{\pgfpoint{\mcSize+\mcThickness}{\mcSize+\mcThickness}}
		\pgfusepath{stroke}
}}
\makeatother


\tikzset{
	pattern size/.store in=\mcSize, 
	pattern size = 5pt,
	pattern thickness/.store in=\mcThickness, 
	pattern thickness = 0.3pt,
	pattern radius/.store in=\mcRadius, 
	pattern radius = 1pt}\makeatletter
\pgfutil@ifundefined{pgf@pattern@name@_gspqjockl}{
	\pgfdeclarepatternformonly[\mcThickness,\mcSize]{_gspqjockl}
	{\pgfqpoint{-\mcThickness}{-\mcThickness}}
	{\pgfpoint{\mcSize}{\mcSize}}
	{\pgfpoint{\mcSize}{\mcSize}}
	{\pgfsetcolor{\tikz@pattern@color}
		\pgfsetlinewidth{\mcThickness}
		\pgfpathmoveto{\pgfpointorigin}
		\pgfpathlineto{\pgfpoint{\mcSize}{0}}
		\pgfpathmoveto{\pgfpointorigin}
		\pgfpathlineto{\pgfpoint{0}{\mcSize}}
		\pgfusepath{stroke}}}
\makeatother


\tikzset{
	pattern size/.store in=\mcSize, 
	pattern size = 5pt,
	pattern thickness/.store in=\mcThickness, 
	pattern thickness = 0.3pt,
	pattern radius/.store in=\mcRadius, 
	pattern radius = 1pt}\makeatletter
\pgfutil@ifundefined{pgf@pattern@name@_qzg8sxqrg}{
	\pgfdeclarepatternformonly[\mcThickness,\mcSize]{_qzg8sxqrg}
	{\pgfqpoint{-\mcThickness}{-\mcThickness}}
	{\pgfpoint{\mcSize}{\mcSize}}
	{\pgfpoint{\mcSize}{\mcSize}}
	{\pgfsetcolor{\tikz@pattern@color}
		\pgfsetlinewidth{\mcThickness}
		\pgfpathmoveto{\pgfpointorigin}
		\pgfpathlineto{\pgfpoint{\mcSize}{0}}
		\pgfpathmoveto{\pgfpointorigin}
		\pgfpathlineto{\pgfpoint{0}{\mcSize}}
		\pgfusepath{stroke}}}
\makeatother


\tikzset{
	pattern size/.store in=\mcSize, 
	pattern size = 5pt,
	pattern thickness/.store in=\mcThickness, 
	pattern thickness = 0.3pt,
	pattern radius/.store in=\mcRadius, 
	pattern radius = 1pt}\makeatletter
\pgfutil@ifundefined{pgf@pattern@name@_bq34wcjhq}{
	\pgfdeclarepatternformonly[\mcThickness,\mcSize]{_bq34wcjhq}
	{\pgfqpoint{-\mcThickness}{-\mcThickness}}
	{\pgfpoint{\mcSize}{\mcSize}}
	{\pgfpoint{\mcSize}{\mcSize}}
	{\pgfsetcolor{\tikz@pattern@color}
		\pgfsetlinewidth{\mcThickness}
		\pgfpathmoveto{\pgfpointorigin}
		\pgfpathlineto{\pgfpoint{\mcSize}{0}}
		\pgfpathmoveto{\pgfpointorigin}
		\pgfpathlineto{\pgfpoint{0}{\mcSize}}
		\pgfusepath{stroke}}}
\makeatother


\tikzset{
	pattern size/.store in=\mcSize, 
	pattern size = 5pt,
	pattern thickness/.store in=\mcThickness, 
	pattern thickness = 0.3pt,
	pattern radius/.store in=\mcRadius, 
	pattern radius = 1pt}
\makeatletter
\pgfutil@ifundefined{pgf@pattern@name@_tnhj0vwuc}{
	\pgfdeclarepatternformonly[\mcThickness,\mcSize]{_tnhj0vwuc}
	{\pgfqpoint{0pt}{0pt}}
	{\pgfpoint{\mcSize+\mcThickness}{\mcSize+\mcThickness}}
	{\pgfpoint{\mcSize}{\mcSize}}
	{
		\pgfsetcolor{\tikz@pattern@color}
		\pgfsetlinewidth{\mcThickness}
		\pgfpathmoveto{\pgfqpoint{0pt}{0pt}}
		\pgfpathlineto{\pgfpoint{\mcSize+\mcThickness}{\mcSize+\mcThickness}}
		\pgfusepath{stroke}
}}
\makeatother


\tikzset{
	pattern size/.store in=\mcSize, 
	pattern size = 5pt,
	pattern thickness/.store in=\mcThickness, 
	pattern thickness = 0.3pt,
	pattern radius/.store in=\mcRadius, 
	pattern radius = 1pt}\makeatletter
\pgfutil@ifundefined{pgf@pattern@name@_gv0f2tvlu}{
	\pgfdeclarepatternformonly[\mcThickness,\mcSize]{_gv0f2tvlu}
	{\pgfqpoint{-\mcThickness}{-\mcThickness}}
	{\pgfpoint{\mcSize}{\mcSize}}
	{\pgfpoint{\mcSize}{\mcSize}}
	{\pgfsetcolor{\tikz@pattern@color}
		\pgfsetlinewidth{\mcThickness}
		\pgfpathmoveto{\pgfpointorigin}
		\pgfpathlineto{\pgfpoint{\mcSize}{0}}
		\pgfpathmoveto{\pgfpointorigin}
		\pgfpathlineto{\pgfpoint{0}{\mcSize}}
		\pgfusepath{stroke}}}
\makeatother
\tikzset{every picture/.style={line width=0.75pt}} 

\begin{tikzpicture}[x=0.75pt,y=0.75pt,yscale=-0.7,xscale=0.7]
	
	\draw  [pattern=_t43tg3toy,pattern size=6pt,pattern thickness=0.75pt,pattern radius=0pt, pattern color={rgb, 255:red, 0; green, 0; blue, 0}] (169,113.8) .. controls (169,72.82) and (239.52,39.6) .. (326.5,39.6) .. controls (413.48,39.6) and (484,72.82) .. (484,113.8) .. controls (484,154.78) and (413.48,188) .. (326.5,188) .. controls (239.52,188) and (169,154.78) .. (169,113.8) -- cycle ;
	\draw  [fill={rgb, 255:red, 255; green, 255; blue, 255 }  ,fill opacity=1 ] (215.87,69.11) .. controls (239.69,58.28) and (346.89,47.44) .. (323.07,69.11) .. controls (299.25,90.78) and (299.25,101.62) .. (323.07,134.12) .. controls (346.89,166.63) and (239.69,166.63) .. (215.87,134.12) .. controls (192.04,101.62) and (192.04,79.95) .. (215.87,69.11) -- cycle ;
	\draw  [fill={rgb, 255:red, 255; green, 255; blue, 255 }  ,fill opacity=1 ] (450.11,110.61) .. controls (474.25,124.2) and (436.55,155.86) .. (408.79,144.4) .. controls (381.03,132.95) and (402.94,118.51) .. (378.72,107.64) .. controls (354.5,96.76) and (395.82,62.97) .. (420.04,73.84) .. controls (444.26,84.71) and (425.97,97.02) .. (450.11,110.61) -- cycle ;
	\draw  [pattern=_gspqjockl,pattern size=2.7750000000000004pt,pattern thickness=0.75pt,pattern radius=0pt, pattern color={rgb, 255:red, 0; green, 0; blue, 0}] (214.1,82.74) .. controls (226.23,76.56) and (251.7,56.55) .. (268.68,82.74) .. controls (285.66,108.92) and (273.53,115.22) .. (268.68,119.79) .. controls (263.82,124.36) and (226.23,138.31) .. (214.1,119.79) .. controls (201.97,101.26) and (201.97,88.91) .. (214.1,82.74) -- cycle ;
	\draw  [pattern=_qzg8sxqrg,pattern size=2.7750000000000004pt,pattern thickness=0.75pt,pattern radius=0pt, pattern color={rgb, 255:red, 0; green, 0; blue, 0}] (271.21,129.12) .. controls (280.17,126.05) and (298.99,116.13) .. (311.54,129.12) .. controls (324.08,142.11) and (315.12,145.23) .. (311.54,147.5) .. controls (307.95,149.77) and (280.17,156.69) .. (271.21,147.5) .. controls (262.25,138.31) and (262.25,132.18) .. (271.21,129.12) -- cycle ;
	\draw  [fill={rgb, 255:red, 255; green, 255; blue, 255 }  ,fill opacity=1 ] (258,83.6) .. controls (263,90.6) and (272.46,103.65) .. (260,111.6) .. controls (247.54,119.55) and (249.5,106.75) .. (233.75,111.67) .. controls (218,116.6) and (212,105.6) .. (222,89.6) .. controls (232,73.6) and (253,76.6) .. (258,83.6) -- cycle ;
	\draw  [pattern=_bq34wcjhq,pattern size=2.7750000000000004pt,pattern thickness=0.75pt,pattern radius=0pt, pattern color={rgb, 255:red, 0; green, 0; blue, 0}] (400.47,94.38) .. controls (389.61,81.83) and (417.61,75.8) .. (426.19,94.38) .. controls (434.76,112.97) and (420.47,109.7) .. (426.19,125.01) .. controls (431.9,140.33) and (406.18,140.33) .. (400.47,125.01) .. controls (394.75,109.7) and (411.33,106.94) .. (400.47,94.38) -- cycle ;
	\draw  [fill=black!25, draw=black] (229,96.8) .. controls (229,92.82) and (234.6,89.6) .. (241.5,89.6) .. controls (248.4,89.6) and (254,92.82) .. (254,96.8) .. controls (254,100.78) and (248.4,104) .. (241.5,104) .. controls (234.6,104) and (229,100.78) .. (229,96.8) -- cycle ;
	\draw  [color={rgb, 255:red, 0; green, 0; blue, 0 }  ,draw opacity=1 ][pattern=_tnhj0vwuc,pattern size=6pt,pattern thickness=0.75pt,pattern radius=0pt, pattern color={rgb, 255:red, 0; green, 0; blue, 0}] (473,12.6) -- (508,12.6) -- (508,25) -- (473,25) -- cycle ;
	\draw  [color={rgb, 255:red, 0; green, 0; blue, 0 }  ,draw opacity=1 ][pattern=_gv0f2tvlu,pattern size=2.7750000000000004pt,pattern thickness=0.75pt,pattern radius=0pt, pattern color={rgb, 255:red, 0; green, 0; blue, 0}] (473,34.6) -- (508,34.6) -- (508,47) -- (473,47) -- cycle ;
	\draw  [fill=black!25, draw=black] (473,56.6) -- (508,56.6) -- (508,69) -- (473,69) -- cycle ;
			\draw (510,6.4) node [anchor=north west][inner sep=0.75pt]  [xscale=1,yscale=1]  {$A_{1}$};
			\draw (510,28.4) node [anchor=north west][inner sep=0.75pt]  [xscale=1,yscale=1]  {$A_{2}$};
			\draw (510,50.4) node [anchor=north west][inner sep=0.75pt]  [xscale=1,yscale=1]  {$A_{3}$};
			\draw (180,156.4) node [anchor=north west][inner sep=0.75pt]  [xscale=1,yscale=1]  {$C$};

		\end{tikzpicture}
		\caption{The construction of \(D_0\).}
		\label{figureconstructionD0}
	\end{figure}
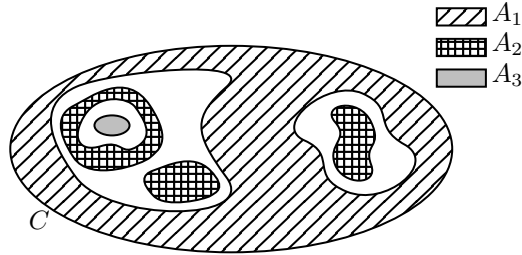
	Now, \(M_+\cup D_0\) is a two-dimensional properly embedded manifold (with corners), which, therefore, separates \(\R^3\) in two connected components. We refer as the \textit{interior} to the one where the mean curvature of \(M_+\) points. Observe that this two-dimensional manifold may be, in principle, unbounded.
	
	Consider \(M\setminus M_+\). It is a non-compact properly embedded complete surface, with non-empty boundary. Its boundary is contained in \(\partial D_0\) and, near the boundary, \(M\setminus M_+\) is contained in \(\R^3_-\). Observe that \(M\setminus M_+\) is, in general, not connected.
	
	We now shift the focus on the set of their boundaries -- that is, on \(\partial D_0\). It is a finite union of disjoint simple loops in \(D\). As above, we call \textit{outer} loop a loop which may be joined to \(C\) via a continuous path in \(D\) which doesn't cross other loops -- there must be some. Observe that, on an outer loop, the mean curvature vector of \(M\) points toward the exterior of the loop. Let $\sigma$ be an outer loop and let \(M_1\) be the connected component of \(M\setminus M_+\) that contains $\sigma$. First, we do some cutting and pasting on \(M_1\) to get a surface that separates \(\R^3\). First, define \(D_1\) as the union of the interiors of the loops which bound \(M_1\), and focus on the loops in $\partial M_1\cup(M_1\cap D_1)$.
	
	For every loop \(\delta\) in \(\partial M_1\), for $\epsilon>0$ sufficiently small, there is a family of loops $\delta^-(\epsilon)$ in \(M_1\cap\{z=-\epsilon\}\) which converge to $\delta$ as $\epsilon$ goes to \(0\). Attach a planar cap to \(M_1\) along the loop \(\delta^-(\epsilon)\) and neglect the annular region of \(M_1\) bounded by $\delta$ and \(\delta^-(\epsilon)\). When several loops of $\partial M_1$ are concentric, a choice of the $\epsilon$'s is possible so that the resulting surface does not have self-intersections.
	
	For a loop $\tilde{\delta}$ in \(M_1\cap D_1\), for $\epsilon>0$ sufficiently small, there is a family of loops $\tilde{\delta}^-(\epsilon)$, \(\tilde{\delta}^+(\epsilon)\) in \(M_1\cap\{z=-\epsilon\}\), \(M_1\cap\{z=\epsilon\}\) respectively which converge to $\sigma$.  Attach two planar caps to \(M_1\) along \(\tilde{\delta}^+(\epsilon)\), \(\tilde{\delta}^-(\epsilon)\) and neglect the annular region of \(M_1\) bounded between \(\tilde{\delta}^+(\epsilon)\) and \(\tilde{\delta}^-(\epsilon)\).
	
	What we get in the end is a properly embedded surface with corners, which may fail to be connected. We focus on the component containing the loop $\sigma$, call it $\tilde{M}_1$. Alexander's duality implies that $\tilde{M}_1$ separates \(\R^3\) in two connected components. 
	
	Now, when \(\tilde{M}_1\) is compact, it is clear that the mean curvature vector must point toward the bounded component of \(\R^3\setminus\tilde{M}_1\). On the other hand, when it is not compact, we know that, on the annular ends, it must point (approximately) toward the axis of rotation -- this comes from the geometric convergence of Theorem \ref{geometricconvergence}.
	
	In any case, since $\sigma$ is outer, the mean curvature vector along $\sigma$ points toward the exterior of $\sigma$, thus \(\tilde{M}_1\cap\mathcal{H}\) must contain an odd number of loops which are non-contractible in $\mathcal{H}\setminus\sigma$. Call $\mathcal{T}_{\text{in}}$ the set of those non-contractible loops which are in \(D\) and $\mathcal{T}_{\text{ext}}$ the one of those contained in $\mathcal{H}\setminus\bar{D}$. We now get a contradiction since, counting the cardinality of $\mathcal{T}_{\text{in}}\cup\mathcal{T}_{\text{ext}}$ in another way, we get that the same number must be even. First of all observe that, since $\sigma$ is outer, we have that the loops in $\mathcal{T}_{\text{in}}$ are contained in \(D_0\).
	
	Indeed, shift the focus on the upper halfspace \(\R^3_+\), and consider the components of \(\tilde{M}_1\cap\R^3_+\). Now, we assumed \(M\) to be transverse to $\mathcal{H}$. Then, the set $\mathcal{T}_{\text{in}}\cup\mathcal{T}_{\text{ext}}$ coincides with the set of the loops of \(\partial(\tilde{M}_1\cap\R^3_+)\) which are non-contractible in $\mathcal{H}\setminus\sigma$. Observe that each connected component of \(\tilde{M}_1\cap\R^3_+\) is contained either in the interior of \(M_+\cup D_0\) or in its exterior. Indeed, components which have a boundary component in \(D_0\) are, close to that boundary component, in the interior of \(M_+\cup D_0\). Moreover, to exit this region, they should pass through \(M_+\), thereby leading to a contradiction by the embeddedness of \(M\). The analogous reasoning holds for components of \(\tilde{M}_1\cap\R^3_+\) which have a boundary component in $\mathcal{H}\setminus\bar{D}$. 
	
	Consider first the components of \(\tilde{M}_1\cap\R^3_+\) in the interior of \(M_+\cup D_0\) -- if any. Since the annular ends cannot be nested, all these components are compact. We aim to use Lemma \ref{lemmacountingloops} to count the cardinality of the components of their boundaries which are non-contractible in $\mathcal{H}\setminus\sigma$ -- that is, to count the cardinality of $\mathcal{T}_{\text{in}}$.
	
	Since $\sigma$ is an outer loop, there exists a point \(p\) in the interior of \(\sigma\), which lies in the exterior of the surface \(M_+\cup D_0\). By cylindrical boundedness, one can then find a proper curve \(r\) homeomorphic to \([0,1)\), contained in the exterior of \(M_+\cup D_0\) and having its endpoint on \(p\). We can then apply Lemma \ref{lemmacountingloops} and get that the cardinality of $\mathcal{T}_{\text{in}}$ is even (when it is not zero).
	
	We now move on to counting the loops in $\mathcal{T}_{\text{ext}}$. Since the cardinality of $\mathcal{T}_{\text{int}}$ is either even or zero, we have that $\mathcal{T}_{\text{ext}}\neq\emptyset$. Consider then the components of \(\tilde{M}_1\cap\R^3_+\) which are in the exterior of \(M_+\cup D_0\).
	
	Remember that we assumed by contradiction \(M_+\) to be non-compact. Then, one can find a properly embedded curve \(r_1\subset M_+\subset\R^3\), homeomorphic to \([0,1)\) and with its endpoint on \(C\). We aim to apply again Lemma \ref{lemmacountingloops}. Now, the components of \(\tilde{M}_1\cap\R^3_+\) need not to be compact. Nevertheless, one can attach a planar cap to each annular end, at a certain height, and neglect the part of the end which sits above the cap. Then, we can apply Lemma \ref{lemmacountingloops} and reach a contradiction, since the number of non-contractible loops -- i.e. the cardinality of $\mathcal{T}_{\text{ext}}$ -- would turn out to be even. Thus, \(M_+\) must be compact.
	
	\textbf{Step 2:} \textit{If \(M\) is not contained in \(\R^3_+\), then \(M_+\) must intersect \(\mathcal{H}\setminus\bar{D}\).}
	\\Assume by contradiction that this is not the case. By the previous step, we know that \(M_+\) is compact. Then, \(M_+\cup D_0\) is a compact cycle, for which the following balancing formula holds \cite{saearpweingarten}:
	\begin{equation}\label{equationflux}
		\int_{D_0}\pe{Y,n_0}=\frac{1}{2}\int_{\partial D_0}\pe{Y,(2a+bT)\nu_0},
	\end{equation}
	where \(Y\) is any constant vector field in \(\R^3\), $\nu_0$ is the co-normal to $\partial D_0$ pointing toward \(M_+\) and \(n_0\) is the unit normal vector field on \(D_0\) which makes \(M_+\cup D_0\) an oriented cycle -- observe that, as an orientation for \(M_+\), we are picking the one determined by its mean curvature vector. We apply \eqref{equationflux} for \(Y=e_3\).
	
	Since, in our assumptions, \(M_+\) does not intersect \(\mathcal{H}\setminus\bar{D}\), we have that \(n_0=e_3\). Now, the operator \(2aI+bT\) is elliptic on \(M\) \cite{saearpweingarten}, therefore its eigenvalues are positive. Thus, we have \(\pe{\nu_0,(2a+bT)\nu_0}>0\) on $\partial D_0$. Since $\nu_0$, by construction, points upwards, this implies that \(\pe{e_3,(2a+bT)\nu_0}>0\) on \(\partial D_0\), so that, by \eqref{equationflux}, we get:
	\begin{equation}\label{equationflux1}
		\int_{C}\pe{e_3,n_0}<\frac{1}{2}\int_{\partial D_0}\pe{e_3,(2a+bT)\nu_0}.
	\end{equation}
	Now, we aim to apply flux formula to a different compact cycle. Start from \(M\cup D\) which is, in general, unbounded. For each annular end \(A_i\), \(i\in I\cup J\), we consider a planar cap \(D_i\) having its boundary on a simple loop \(\partial D_i\) of \(A_i\) -- this is possible by the aforementioned geometric convergence of every annular end to a W-Delaunay surface. We attach this planar cap to \(M\) and we neglect the part of \(A_i\) that sits above \(\partial D_i\) -- or below, in the case of negative ends. What we get is a compact cycle, for which the balancing formula, with \(Y=e_3\), reads as:
	
	\begin{equation}\label{equationflux2}
		\int_{D_0}\pe{e_3,n_0}+\sum_{i\in I\cup J}\int_{D_i}\pe{e_3,n_{i}}=\frac{1}{2}\int_{\partial D_0}\pe{e_3,(2a+bT)\nu_0}+\frac{1}{2}\sum_{i\in I\cup J}\int_{\partial D_i}\pe{e_3,(2a+bT)\nu_i},
	\end{equation}
	where $\nu_i$ is the co-normal to $\partial D_i$ pointing downwards for \(i\in I\) and upwards for \(i\in J\). Moreover, the vectors \(n_i\) are chosen so that, with the orientation of \(M\) given by the mean curvature vector, the compact cycle obtained is oriented. We will apply it for \(Y\) being equal to \(e_3=(0,0,1)\). Observe that \(n_i\) is equal to \(e_3\) for \(i\in I\), it is equal to \(-e_3\) for \(i\in J\), and \(n_0=e_3\), where the last thing comes from the fact that, by our assumption, we have \(M\cap (\mathcal{H}\setminus\bar{D})=\emptyset\). Rearranging the terms, we get:
	
	\begin{equation}\label{balancing1}
		\frac{1}{2}\int_{\partial D_0}\pe{e_3,(2a+bT)\nu_0}=|D_0|+\sum_{i\in I\cup J}\left[\int_{D_i}\pe{e_3,n_{i}}-\frac{1}{2}\int_{\partial D_i}\pe{e_3,(2a+bT)\nu_i}\right]	.	
	\end{equation}
	Regarding the terms appearing in the sum, observe, as in the proof of Lemma \ref{lemmamassadelaunay}, that the quantity
	\begin{equation}\label{equationquantity}
		\int_{D_i}\pe{e_3,n_{i}}-\frac{1}{2}\int_{\partial D_i}\pe{e_3,(2a+bT)\nu_i}
	\end{equation}
	is homotopically invariant. That is, for each non-trivial loop $\Gamma$ homotopic to\,$\partial D_i$ and each compact surface \(K\) with $\partial K\equiv \Gamma$, it holds:
	\[\int_{D_i}\pe{e_3,n_{i}}-\frac{1}{2}\int_{\partial D_i}\pe{e_3,(2a+bT)\nu_i}=\int_{K}\pe{e_3,n_{K}}-\frac{1}{2}\int_{\Gamma}\pe{e_3,(2a+bT)\nu_\Gamma},\]
	where \(n_K\) is the normal to \(K\) and $\nu_\Gamma$ is the conormal to $\Gamma$, pointing accordingly to\,\(n_i\),\,\(\nu_i\) respectively. Now, consider a positive end. By Theorem \ref{geometricconvergence}, we have that such an annular end converges geometrically to a W-Delaunay. Thus, by choosing $\Gamma$ higher and higher, we get that \eqref{equationquantity} must equal the analogous quantity computed on the W-Delaunay surface that the end converges to. Lemma \ref{lemmamassadelaunay} then implies that it is equal to \(\pi(R_ir_i+b)\), where \(R_i\) and \(r_i\) are the big and small radii of the W-Delaunay to which the end converge. For a negative end, we get a negative sign since the orientation of \(n_i\) and \(\nu_i\) is the opposite.
	Thus, \eqref{balancing1} becomes:
	\begin{equation}\label{equationflux3}
		\frac{1}{2}\int_{C}\pe{e_3,(2a+bT)\nu_0}=|D|+\sum_{i\in I}\mathcal{W}_i-\sum_{j\in J}\mathcal{W}_j,
	\end{equation}
	By comparing \eqref{equationflux1} with \eqref{equationflux3} we now get a contradiction:
	\[0<|D|-|D_0|\leq\sum_{j\in J}\mathcal{W}_j-\sum_{i\in I}\mathcal{W}_i,\]
	since, by \eqref{equationbalance}, we have that the term on the right-hand side is non-positive.
	Thus, when \(M\) is not contained in \(\R^3_+\), we get that \(M_+\) must intersect $\mathcal{H}\setminus\bar{D}$.

\textbf{Step 3:} \textit{When \(M\) is not contained in \(\R^3_+\), there exists a unique non-contractible loop $\gamma$ in \(\bar{M}^+\cap(\mathcal{H}\setminus\bar{D})\). In particular, the mean curvature vector of \(M\), near \(C\), must point toward the exterior of \(C\).}
\\The fact that the existence and uniqueness of $\gamma$ implies a restriction on the direction of the mean curvature vector of \(M\) near \(C\) is immediate. The existence and uniqueness of $\gamma$ is proved by Alexandrov reflection principle. To start the method, we have to do some cutting and pasting, to obtain a manifold that separates \(\R^3\).  Observe that \(M\cap D\) is a finite union of disjoint simple loops. For each of this loops \(C_j\), for $\epsilon>0$ sufficiently small, there is a family of loops \(C_j^+(\epsilon)\) and \(C_j^-(\epsilon)\) in \(M\cap\{z=\epsilon\}\) and \(M\cap\{z=-\epsilon\}\) respectively, which converge to \(C_j\). For each \(C_j\) we pick a different value of $\epsilon$, we attach two planar caps to \(C_j^+(\epsilon)\) and \(C_j^-(\epsilon)\) and we neglect the annular region of \(M\) between the two loops. Moreover, we attach the planar cap \(D\) to \(M\) along \(C\). There is a choice of the $\epsilon$'s such that the resulting surface is an embedded, not necessarily connected, surface with corners. We consider only the component that contains \(C\), and we call it \(\tilde{M}\). This is a properly embedded connected surface (with corners), and so it separates \(\R^3\) in two connected components. By our assumptions, we have that $\tilde{M}$ must intersect $\mathcal{H}$ also outside of \(\bar{D}\). In particular, we have that \(\bar{M}_+\cap(\mathcal{H}\setminus\bar{D})\subset\tilde{M}\cap(\mathcal{H}\setminus\bar{D})\). We show that $\tilde{M}\cap(\mathcal{H}\setminus\bar{D})$ contains a unique non-contractible loop $\gamma$, thereby implying -- since \(M\), in a neighborhood of \(C\), is contained in \(\R^3_+\) -- that\,\,\(\bar{M}_+\cap(\mathcal{H}\setminus\bar{D})=\tilde{M}\cap(\mathcal{H}\setminus\bar{D})\), and concluding the proof of Step 3.

Now, $\tilde{M}\cap(\mathcal{H}\setminus\bar{D})$ is the union of a finite number of disjoint loops. To show the existence of $\gamma$, we show that the first loop that a vertical plane meets, coming from infinity toward \(C\), cannot be trivial in \(\mathcal{H}\setminus\bar{D}\). Indeed, suppose by contradiction that this is not the case, and let $\beta$ be the trivial loop maximizing the distance from \(C\). Let \(Q(t)\) be a family of vertical parallel planes, with normal $\nu$, coming from infinity toward \(C\), first touching $\beta$ on the farthest point from \(C\).
	  	  	  \begin{figure}[!htb]

	\tikzset{every picture/.style={line width=0.65pt}} 
	\centering
	\begin{tikzpicture}[x=0.75pt,y=0.75pt,yscale=-0.7,xscale=0.7]
		
		\draw    (226,77.2) -- (178,119.6) ;
		\draw    (226,161.2) -- (178,203.6) ;
		\draw    (226,245.2) -- (178,287.6) ;
		\draw    (178,119.6) -- (454,119.6) ;
		\draw    (178,203.6) -- (454,203.6) ;
		\draw    (178,287.6) -- (454,287.6) ;
		\draw    (502,161.2) -- (454,203.6) ;
		\draw    (502,245.2) -- (454,287.6) ;
		\draw    (502,77.2) -- (454,119.6) ;
		\draw   (222,178.8) .. controls (222,173.72) and (240.19,169.6) .. (262.63,169.6) .. controls (285.07,169.6) and (303.26,173.72) .. (303.26,178.8) .. controls (303.26,183.88) and (285.07,188) .. (262.63,188) .. controls (240.19,188) and (222,183.88) .. (222,178.8) -- cycle ;
		\draw   (432.68,163.38) .. controls (441.82,167.03) and (418.45,167.43) .. (412,170.6) .. controls (405.55,173.77) and (418.38,183.81) .. (407,185.6) .. controls (395.62,187.39) and (348.12,182.57) .. (368.12,168.86) .. controls (388.11,155.15) and (423.53,159.72) .. (432.68,163.38) -- cycle ;
		\draw   (252,94.3) .. controls (252,91.15) and (262.07,88.6) .. (274.5,88.6) .. controls (286.93,88.6) and (297,91.15) .. (297,94.3) .. controls (297,97.45) and (286.93,100) .. (274.5,100) .. controls (262.07,100) and (252,97.45) .. (252,94.3) -- cycle ;
		\draw   (308,79.3) .. controls (308,76.15) and (318.07,73.6) .. (330.5,73.6) .. controls (342.93,73.6) and (353,76.15) .. (353,79.3) .. controls (353,82.45) and (342.93,85) .. (330.5,85) .. controls (318.07,85) and (308,82.45) .. (308,79.3) -- cycle ;
		\draw   (332,99.8) .. controls (332,97.48) and (339.42,95.6) .. (348.58,95.6) .. controls (357.74,95.6) and (365.16,97.48) .. (365.16,99.8) .. controls (365.16,102.12) and (357.74,104) .. (348.58,104) .. controls (339.42,104) and (332,102.12) .. (332,99.8) -- cycle ;
		\draw   (245,263.3) .. controls (245,260.15) and (255.07,257.6) .. (267.5,257.6) .. controls (279.93,257.6) and (290,260.15) .. (290,263.3) .. controls (290,266.45) and (279.93,269) .. (267.5,269) .. controls (255.07,269) and (245,266.45) .. (245,263.3) -- cycle ;
		\draw   (325,272.3) .. controls (325,269.15) and (335.07,266.6) .. (347.5,266.6) .. controls (359.93,266.6) and (370,269.15) .. (370,272.3) .. controls (370,275.45) and (359.93,278) .. (347.5,278) .. controls (335.07,278) and (325,275.45) .. (325,272.3) -- cycle ;
		\draw   (324,250.8) .. controls (324,248.48) and (331.42,246.6) .. (340.58,246.6) .. controls (349.74,246.6) and (357.16,248.48) .. (357.16,250.8) .. controls (357.16,253.12) and (349.74,255) .. (340.58,255) .. controls (331.42,255) and (324,253.12) .. (324,250.8) -- cycle ;
		\draw   (381,261.8) .. controls (381,259.48) and (388.42,257.6) .. (397.58,257.6) .. controls (406.74,257.6) and (414.16,259.48) .. (414.16,261.8) .. controls (414.16,264.12) and (406.74,266) .. (397.58,266) .. controls (388.42,266) and (381,264.12) .. (381,261.8) -- cycle ;
		\draw [color={rgb, 255:red, 208; green, 2; blue, 27 }  ,draw opacity=1 ]   (440,64.6) -- (440,320.6) ;
		\draw [color={rgb, 255:red, 208; green, 2; blue, 27 }  ,draw opacity=1 ]   (488,278.2) -- (440,320.6) ;
		\draw [color={rgb, 255:red, 208; green, 2; blue, 27 }  ,draw opacity=1 ]   (488,22.2) -- (440,64.6) ;
		\draw [color={rgb, 255:red, 208; green, 2; blue, 27 }  ,draw opacity=1 ] [dash pattern={on 4.5pt off 4.5pt}]  (488,77.2) -- (440,119.6) ;
		\draw [color={rgb, 255:red, 208; green, 2; blue, 27 }  ,draw opacity=1 ] [dash pattern={on 4.5pt off 4.5pt}]  (488,161.2) -- (440,203.6) ;
		\draw [color={rgb, 255:red, 208; green, 2; blue, 27 }  ,draw opacity=1 ] [dash pattern={on 4.5pt off 4.5pt}]  (488,245.2) -- (440,287.6) ;
		\draw [color={rgb, 255:red, 208; green, 2; blue, 27 }  ,draw opacity=1 ]   (452,76.4) -- (417,76.4) ;
		\draw [shift={(415,76.4)}, rotate = 360] [color={rgb, 255:red, 208; green, 2; blue, 27 }  ,draw opacity=1 ][line width=0.75]    (10.93,-3.29) .. controls (6.95,-1.4) and (3.31,-0.3) .. (0,0) .. controls (3.31,0.3) and (6.95,1.4) .. (10.93,3.29)   ;
		
		\draw (282,153.4) node [anchor=north west][inner sep=0.75pt]  [xscale=1,yscale=1]  {$C$};
		\draw (421,139.4) node [anchor=north west][inner sep=0.75pt]  [xscale=1,yscale=1]  {$\beta $};
		\draw (152,200) node [anchor=north west][inner sep=0.75pt]  [xscale=1,yscale=1]  {$\mathcal{H}$};
		\draw (152,118.4) node [anchor=north west][inner sep=0.75pt]  [xscale=1,yscale=1]  {$\mathcal{H}^+_{n}$};
		\draw (152,285.4) node [anchor=north west][inner sep=0.75pt]  [xscale=1,yscale=1]  {$\mathcal{H}^-_{n}$};
		\draw (400,304.4) node [anchor=north west][inner sep=0.75pt]  [xscale=1,yscale=1]  {$Q( t)$};
		\draw (424,55.4) node [anchor=north west][inner sep=0.75pt]  [xscale=1,yscale=1]  {$\nu $};

	\end{tikzpicture}
	\caption{The Alexandrov procedure.}
	\label{figureAlexandrov}
\end{figure}
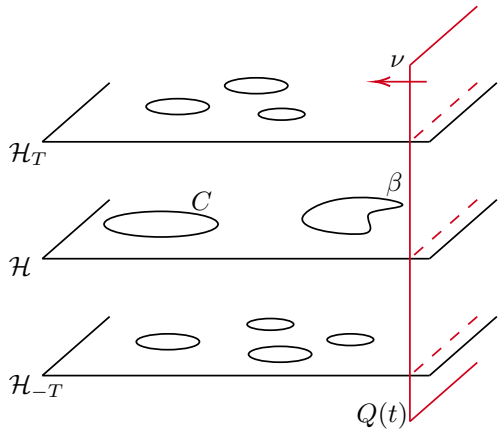

By the proof of Lemma \ref{lemmacircles}, we know there exists a positive diverging sequence \(\{t_n\}_{n\in\N}\) and some values \(\{T_i\}_{i\in I\cup J}\) such that, for \(i\in I\), the intersection of the \(i\)-th end with the plane $\{z=t_n+T_i\}$ is a simple closed loop converging smoothly to a circle. Moreover, given a vertical plane \(Q\), we have the convergence of the corresponding Alexandrov function as in \eqref{equationconvergencealpha}. The analogous statement holds for the \(j\)-th end, for \(j\in J\), with the plane \(\{z=-t_n+T_j\}\).

The strategy would be now to cut every positive \(i\)-th end at height \(t_n+T_i\) (and every negative \(j\)-th end ad height \(-t_n+T_j\)), and forget the part of the end which sits above such a height (or below, for negative ends); then, use Alexandrov method on the resulting compact surface. In the following, purely for notational convenience, we will assume that all the \(T_i\)'s are zero, for \(i\in I\cup J\), and denote the plane\,\,\(\{z=\pm t_n\}\) by $\mathcal{H}^\pm_n$. This assumption entails no loss of generality and does not affect the proof, since the argument in the general case is identical (up to a straightforward translation).

We initially assume that no end of \(\tilde{M}\) has a plane of symmetry in the direction of\,$\nu$. We pick an initial plane disjoint from \(\tilde{M}\) -- this is possible by cylindrical boundedness -- and we start moving it toward \(\tilde{M}\). We only focus on what happens on the loops in the planes \(\mathcal{H}^\pm_{n}\) and $\beta$. We stop when the Alexandrov procedure stops for either $\beta$ or the loops in \(\mathcal{H}^\pm_{n}\) (Figure \ref{figureAlexandrov}). This shall happen before the planes have reached \(C\). In any case, Lemma\,\,\ref{lemmaalphafunction1} implies that, in the slab\,\,\(\{|z|\leq t_n\}\), an interior touching point for \(\tilde{M}\) must have occurred. Indeed, apply Lemma \ref{lemmaalphafunction1} separately on each annular end, for \(|z|\geq t_n\). Then, outside of this horizontal slab, the monotonicity of the Alexandrov functions relative to each end implies that the reflection of \(\tilde{M}\) through the planes is still contained in the interior of $\tilde{M}$. But this is absurd: we would have found a plane of symmetry for \(M\), with \(C\) all in one side; \(\tilde{M}\cap\{|z|\leq t_n\}\) being a graph in the other halfspace.

\begin{figure}[!htb]

	\centering
	
	\tikzset{every picture/.style={line width=0.6pt}} 
	
	\begin{tikzpicture}[x=0.75pt,y=0.75pt,yscale=-0.65,xscale=0.65]
		
		\draw   (107,227.3) .. controls (107,219.18) and (134.31,212.6) .. (168,212.6) .. controls (201.69,212.6) and (229,219.18) .. (229,227.3) .. controls (229,235.42) and (201.69,242) .. (168,242) .. controls (134.31,242) and (107,235.42) .. (107,227.3) -- cycle ;
		\draw    (160,6.6) .. controls (160.36,16.66) and (176.11,29.28) .. (176,41.6) .. controls (175.89,53.92) and (162.44,59.59) .. (161,74.6) .. controls (159.56,89.61) and (177,93.6) .. (177,108.6) .. controls (177,123.6) and (165.7,129.01) .. (166,139.6) .. controls (166.3,150.19) and (171,170.6) .. (187,170.6) .. controls (203,170.6) and (225.97,159.43) .. (242,144.6) .. controls (258.03,129.77) and (241,120.6) .. (240,108.6) .. controls (239,96.6) and (251,85.6) .. (252,74.6) .. controls (253,63.6) and (241,57.6) .. (241,41.6) .. controls (241,25.6) and (251,15.82) .. (251,8.6) ;
		\draw  [dash pattern={on 4.5pt off 4.5pt}]  (262,4.6) -- (262,271.6) ;
		\draw    (35,294.6) .. controls (35.07,291.83) and (20.21,282.17) .. (21,269.6) .. controls (21.79,257.03) and (40,255.6) .. (41,244.6) .. controls (42,233.6) and (25.98,226.68) .. (29,215.6) .. controls (32.02,204.52) and (37,191.6) .. (62,184.6) .. controls (87,177.6) and (138,174.6) .. (139,147.6) .. controls (140,120.6) and (118,121.6) .. (118,106.6) .. controls (118,91.6) and (135,84.6) .. (136,69.6) .. controls (137,54.6) and (123.27,55.06) .. (124,42.6) .. controls (124.73,30.14) and (138.01,18.02) .. (142,8) ;
		\draw    (59,292.6) .. controls (58.6,287.99) and (72,281.6) .. (74,270.6) .. controls (76,259.6) and (55.2,253.61) .. (59,244.6) .. controls (62.8,235.59) and (78.62,234.69) .. (83,225.6) .. controls (87.38,216.51) and (81.22,199) .. (88,195.6) .. controls (94.78,192.2) and (108.96,215.92) .. (107,227.3) ;
		\draw  [color={rgb, 255:red, 208; green, 2; blue, 27 }  ,draw opacity=1 ] (416.74,227.01) .. controls (416.74,218.91) and (389.48,212.34) .. (355.84,212.34) .. controls (322.21,212.34) and (294.95,218.91) .. (294.95,227.01) .. controls (294.95,235.12) and (322.21,241.69) .. (355.84,241.69) .. controls (389.48,241.69) and (416.74,235.12) .. (416.74,227.01) -- cycle ;
		\draw [color={rgb, 255:red, 208; green, 2; blue, 27 }  ,draw opacity=1 ]   (363.83,6.68) .. controls (363.47,16.72) and (347.74,29.32) .. (347.86,41.62) .. controls (347.97,53.92) and (361.4,59.58) .. (362.83,74.57) .. controls (364.27,89.55) and (346.86,93.53) .. (346.86,108.51) .. controls (346.86,123.48) and (358.14,128.89) .. (357.84,139.46) .. controls (357.54,150.03) and (352.85,170.41) .. (336.88,170.41) .. controls (320.9,170.41) and (297.97,159.25) .. (281.97,144.45) .. controls (265.97,129.64) and (282.97,120.49) .. (283.96,108.51) .. controls (284.96,96.53) and (272.98,85.55) .. (271.98,74.57) .. controls (270.99,63.58) and (282.97,57.59) .. (282.97,41.62) .. controls (282.97,25.65) and (272.98,15.88) .. (272.98,8.68) ;
		\draw [color={rgb, 255:red, 208; green, 2; blue, 27 }  ,draw opacity=1 ]   (488.62,294.2) .. controls (488.56,291.43) and (503.38,281.79) .. (502.6,269.24) .. controls (501.81,256.69) and (483.63,255.26) .. (482.63,244.28) .. controls (481.63,233.3) and (497.63,226.39) .. (494.61,215.33) .. controls (491.6,204.27) and (486.63,191.37) .. (461.67,184.38) .. controls (436.71,177.39) and (385.79,174.4) .. (384.8,147.44) .. controls (383.8,120.49) and (405.76,121.49) .. (405.76,106.51) .. controls (405.76,91.54) and (388.79,84.55) .. (387.79,69.57) .. controls (386.79,54.6) and (400.5,55.06) .. (399.77,42.62) .. controls (399.04,30.18) and (385.79,18.08) .. (381.8,8.08) ;
		\draw [color={rgb, 255:red, 208; green, 2; blue, 27 }  ,draw opacity=1 ]   (464.66,292.2) .. controls (465.06,287.6) and (451.68,281.22) .. (449.69,270.24) .. controls (447.69,259.26) and (468.45,253.28) .. (464.66,244.28) .. controls (460.87,235.29) and (445.07,234.39) .. (440.7,225.31) .. controls (436.33,216.24) and (442.48,198.75) .. (435.71,195.36) .. controls (428.94,191.97) and (414.79,215.65) .. (416.74,227.01) ;
		\draw   (229.08,228) .. controls (229.03,226.42) and (229,224.81) .. (229,223.2) .. controls (229,186.36) and (243.76,156.5) .. (261.97,156.5) .. controls (262.29,156.5) and (262.6,156.51) .. (262.92,156.53) ;  
		\draw  [color={rgb, 255:red, 208; green, 2; blue, 27 }  ,draw opacity=1 ] (262.01,156.5) .. controls (280.2,156.54) and (294.95,186.39) .. (294.95,223.2) .. controls (294.95,224.27) and (294.93,225.34) .. (294.91,226.4) ;  
		\draw [color={rgb, 255:red, 126; green, 211; blue, 33 }  ,draw opacity=1 ]   (294.95,227.01) .. controls (294,274.4) and (340,262.4) .. (343,285.4) ;
		\draw [color={rgb, 255:red, 126; green, 211; blue, 33 }  ,draw opacity=1 ]   (416.74,227.01) .. controls (416.36,245.93) and (404.63,242.73) .. (394,249.6) .. controls (383.37,256.47) and (391.05,261.66) .. (385,267.6) .. controls (378.95,273.54) and (364.07,277.49) .. (365,284.6) ;
		\draw  [dash pattern={on 4.5pt off 4.5pt}] (118,105.61) .. controls (118,101.63) and (130.98,98.41) .. (147,98.41) .. controls (163.02,98.41) and (176,101.63) .. (176,105.61) .. controls (176,109.58) and (163.02,112.8) .. (147,112.8) .. controls (130.98,112.8) and (118,109.58) .. (118,105.61) -- cycle ;
		\draw  [color={rgb, 255:red, 208; green, 2; blue, 27 }  ,draw opacity=1 ][dash pattern={on 4.5pt off 4.5pt}] (347.76,106.51) .. controls (347.76,102.54) and (360.74,99.32) .. (376.76,99.32) .. controls (392.78,99.32) and (405.76,102.54) .. (405.76,106.51) .. controls (405.76,110.49) and (392.78,113.71) .. (376.76,113.71) .. controls (360.74,113.71) and (347.76,110.49) .. (347.76,106.51) -- cycle ;
		\draw  [color={rgb, 255:red, 208; green, 2; blue, 27 }  ,draw opacity=1 ][dash pattern={on 4.5pt off 4.5pt}] (262.06,99.2) .. controls (262.21,99.2) and (262.35,99.2) .. (262.5,99.2) .. controls (274.37,99.2) and (284,102.24) .. (284,106) .. controls (284,109.71) and (274.63,112.72) .. (262.97,112.8) ;  
		\draw  [dash pattern={on 4.5pt off 4.5pt}] (261.95,112.8) .. controls (250.33,112.71) and (241,109.7) .. (241,106) .. controls (241,102.29) and (250.39,99.27) .. (262.06,99.2) ;  
		\draw  [dash pattern={on 4.5pt off 4.5pt}] (21,270.7) .. controls (21,267.33) and (32.86,264.6) .. (47.5,264.6) .. controls (62.14,264.6) and (74,267.33) .. (74,270.7) .. controls (74,274.06) and (62.14,276.79) .. (47.5,276.79) .. controls (32.86,276.79) and (21,274.06) .. (21,270.7) -- cycle ;
		\draw  [color={rgb, 255:red, 208; green, 2; blue, 27 }  ,draw opacity=1 ][dash pattern={on 4.5pt off 4.5pt}] (449.69,270.6) .. controls (449.69,267.23) and (461.55,264.5) .. (476.19,264.5) .. controls (490.82,264.5) and (502.69,267.23) .. (502.69,270.6) .. controls (502.69,273.97) and (490.82,276.7) .. (476.19,276.7) .. controls (461.55,276.7) and (449.69,273.97) .. (449.69,270.6) -- cycle ;
		\draw  [color={rgb, 255:red, 126; green, 211; blue, 33 }  ,draw opacity=1 ][dash pattern={on 4.5pt off 4.5pt}] (323,265.8) .. controls (323,261.82) and (337.1,258.6) .. (354.5,258.6) .. controls (371.9,258.6) and (386,261.82) .. (386,265.8) .. controls (386,269.78) and (371.9,273) .. (354.5,273) .. controls (337.1,273) and (323,269.78) .. (323,265.8) -- cycle ;
		
		\draw (266,259.4) node [anchor=north west][inner sep=0.75pt]  [xscale=1,yscale=1]  {$Q$};
		\draw (170,216) node [anchor=north west][inner sep=0.75pt]  [xscale=1,yscale=1]  {$C$};
		\draw (357.84,215.74) node [anchor=north west][inner sep=0.75pt]  [xscale=1,yscale=1]  {$C^{*}$};

	\end{tikzpicture}
	\caption{An end with a plane of symmetry, with \(C\) all in one side.}
	\label{figure2}
\end{figure}
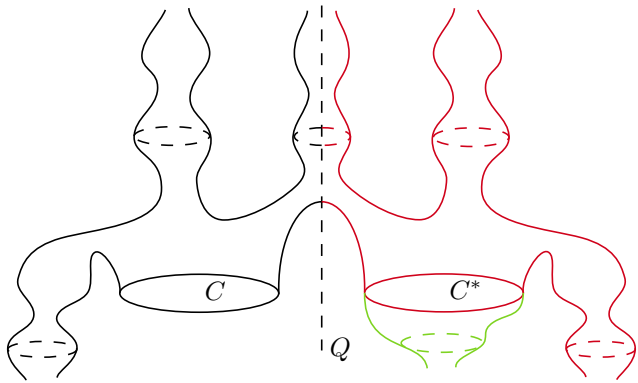

What happens when there is an end with a plane of symmetry in the direction\,\,$\nu$? Clearly, we cannot use the monotonicity of the Alexandrov function, on that end. We now see that an interior touching point must have occurred strictly before the Alexandrov procedure stops for the loop relative to that end. First observe that, when the Alexandrov procedure stops for a touching point for a loop in $\mathcal{H}^\pm_{n}$, we can assume -- up to the choice of a bigger\,\(n\) -- that it has stopped arbitrarily close to the axis of the end relative to the loop. This comes from Lemma \ref{lemmacircles} and Theorem\,\,\ref{geometricconvergence}.

Now, assume \(\tilde{M}\) has at least one end with symmetry in the direction $\nu$, and consider the one whose axis is the first one to be touched by the planes \(Q(t)\) -- this end may be not unique, but still everything we do holds. We call it \(A\). Let $\sigma$ be the loop relative to \(A\). Now, if the axis of \(A\) passes through \(D\), then, by what we have just said, up to the choice of a bigger \(n\) we have that the Alexandrov procedure stops for $\beta$, or for a loop relative to an end without symmetry, strictly before it can stop for $\sigma$.

The only case left is when the axis of \(A\) does not pass through \(D\). In this case, the symmetry of \(M\) implies the existence of another trivial loop -- i.e. \(C^*\), the reflection of \(C\) through the plane of symmetry. Indeed, it follows from the symmetry of \(A\) that the component of \(M\cap\R^3_+\) containing \(C\) is also symmetric, with respect to the same plane (check Figure \ref{figure2}). Now, it may happen that \(\beta\equiv C^*\), or $\beta$ is a non-contractible or trivial loop in $\mathcal{H}\setminus \overline{D^*}$.

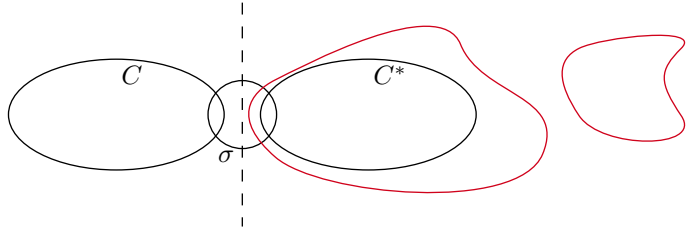
\begin{figure}[!htb]
	\centering

	\tikzset{every picture/.style={line width=0.5pt}} 
	
	\begin{tikzpicture}[x=0.75pt,y=0.75pt,yscale=-0.65,xscale=0.65]
		
		\draw   (95,121.96) .. controls (95,98.28) and (132.06,79.09) .. (177.79,79.09) .. controls (223.51,79.09) and (260.57,98.28) .. (260.57,121.96) .. controls (260.57,145.63) and (223.51,164.82) .. (177.79,164.82) .. controls (132.06,164.82) and (95,145.63) .. (95,121.96) -- cycle ;
		\draw   (248.19,121.96) .. controls (248.19,107.43) and (259.97,95.65) .. (274.5,95.65) .. controls (289.03,95.65) and (300.81,107.43) .. (300.81,121.96) .. controls (300.81,136.48) and (289.03,148.26) .. (274.5,148.26) .. controls (259.97,148.26) and (248.19,136.48) .. (248.19,121.96) -- cycle ;
		\draw   (288.43,121.96) .. controls (288.43,98.28) and (325.49,79.09) .. (371.21,79.09) .. controls (416.94,79.09) and (454,98.28) .. (454,121.96) .. controls (454,145.63) and (416.94,164.82) .. (371.21,164.82) .. controls (325.49,164.82) and (288.43,145.63) .. (288.43,121.96) -- cycle ;
		\draw  [dash pattern={on 4.5pt off 4.5pt}]  (274.5,35.3) -- (274.5,208.61) ;
		\draw  [color={rgb, 255:red, 208; green, 2; blue, 27 }  ,draw opacity=1 ] (300,95) .. controls (320,85) and (426,27.6) .. (442,67.6) .. controls (458,107.6) and (524,106.6) .. (505,152.6) .. controls (486,198.6) and (334,184.6) .. (300,155) .. controls (266,125.4) and (280,105) .. (300,95) -- cycle ;
		\draw  [color={rgb, 255:red, 208; green, 2; blue, 27 }  ,draw opacity=1 ] (533.06,70.66) .. controls (550.48,61.95) and (628.86,53.24) .. (611.44,70.66) .. controls (594.02,88.07) and (594.02,96.78) .. (611.44,122.91) .. controls (628.86,149.03) and (550.48,149.03) .. (533.06,122.91) .. controls (515.65,96.78) and (515.65,79.36) .. (533.06,70.66) -- cycle ;
		
		\draw (179.79,82.49) node [anchor=north west][inner sep=0.75pt]  [xscale=1,yscale=1]  {$C$};
		\draw (373.21,82.49) node [anchor=north west][inner sep=0.75pt]  [xscale=1,yscale=1]  {$C^{*}$};
		\draw (254,148) node [anchor=north west][inner sep=0.75pt]  [xscale=1,yscale=1]  {$\sigma $};

	\end{tikzpicture}
	\caption{How \(C^*\) and $\beta$ may be arranged.}
	\label{figurearrangement}
\end{figure}

In any case, the existence of \(C^*\) acts as a ``barrier'' for the Alexandrov procedure with direction $\nu$ (check Figure \ref{figurearrangement}). Indeed, the Alexandrov procedure must have stopped -- for a touching point of $\beta$ with itself or for a touching point of $\beta$ with\,\,\(C^*\), when \(\beta\not\equiv C^*\) -- before the plane has swept the whole \(C^*\) and, a fortiori, before a touching point for $\sigma$ -- again, up to the choice of a bigger \(n\). So there exists a non-contractible loop $\gamma$ in $\mathcal{H}\setminus\bar{D}$, such that \(\tilde{M}\) does not touch $\mathcal{H}$ in the exterior of\,\,$\gamma$.

The fact that $\gamma$ is the unique non-contractible loop in $\mathcal{H}\setminus\bar{D}$ is proved analogously to what we did with the non-existence of $\beta$. Namely, we assume by contradiction that there is another non-contractible loop \(\gamma_1\), in the interior of $\gamma$, and use the Alexandrov reflection principle, morally looking for touching points of the reflections of $\gamma$ with $\gamma_1$. Thus, $\gamma$ is unique.

As anticipated, the existence and uniqueness of $\gamma$ implies that the mean curvature vector along \(C\) points toward the exterior of \(C\). This concludes Step 3.

Now, to conclude the proof of Theorem \ref{two} in the case of positive balance, we apply flux formula to \(M\), as in the computation of \eqref{equationflux3}. Nevertheless, in this case, we know that the normal vector to\,\,\(D\) is \(-e_3\), and so:
					\begin{equation}\label{equation120925}
						0<\frac{1}{2}\int_{C}\pe{e_3,(2a+bT)\nu_0}=-|D|+\sum_{i\in I}\mathcal{W}_i-\sum_{j\in J}\mathcal{W}_j	,
					\end{equation}
which concludes the proof in the case of positive balance.

\textbf{The case of negative balance:} The proof in this case essentially follows from what was done above. We first claim that \(\vec{H}\), along \(C\), points toward the interior of \(C\). Indeed, if this was not the case, then, by flux formula, we would get \eqref{equation120925}, which would imply\,\,\(|D|<0\).

Our next claim is that \(M_+\), as defined above, doesn't intersect \(\mathcal{H}\setminus\bar{D}\). Indeed, if \(\mathcal{H}\setminus\bar{D}\neq\emptyset\), then we could apply Step 3 above, and get a contradiction by the preceding claim.

Now, \(M\) cannot be contained in the upper halfspace, since there must be at least one negative end. Thus, it holds \(M_+\cap D\neq\emptyset\) and \(M_+\cap(\mathcal{H}\setminus\bar{D})=\emptyset\). We can then apply Step 1 above, and get, in particular, that \(\mathcal{T}_{\text{ext}}\neq\emptyset\), which implies the existence of the loop $\gamma$ in the thesis. The bound on \(|D|\) now follows from flux formula, taking into consideration the direction of $\vec{H}$ along \(C\). This concludes the proof.
\end{proof}

As in \cite{saearpconvex}, we have the following corollary.
\begin{corollary}\label{corollary}
	Assume \(M\) satisfies the hypothesis of Theorem \ref{two} with: \[\sum_{i\in I}\mathcal{W}_i-\sum_{j\in J}\mathcal{W}_j\geq0,\] it is non-compact and $\partial M$ is a circle of radius \(r\). Then, either \(M\) is an annular end of a W-Delaunay surface or \(M\) has at least\,\,$r^2(2a^2+b)^{-1}$ positive ends.
\end{corollary}
\begin{proof}
	It is a straightforward application of Theorem \ref{zero} and Theorem \ref{two}, using the fact that, by \eqref{relation}, we have that \(R_jr_j<2a^2\) for each \(j\). 
\end{proof}
As already mentioned, we remark that the analogue of Theorem \ref{two} holds in the cmc case, thereby extending the result of Rosenberg and Sa Earp in \cite{saearpconvex}. In this setting, the mass of an annular end \(E\), converging to a Delaunay surface \(D\), is defined as \(\mathcal{W}\defeq\pi\left(\frac{r}{H}-r^2\right)\), where \(r\) is the small radius of \(D\) and \(H\) is the value of the constant mean curvature.
\begin{theorem}\label{twocmc}
	Let \(M\) be a complete, properly embedded surface, having non-zero constant mean curvature \(H\), with boundary a strictly convex planar curve\,\,\(C\subset\mathcal{H}\). Let \(D\) be the interior of \(C\). Assume that \(M\) is transverse to \(\mathcal{H}\) along\,\,\(C\) and \(M\) is contained, say, in the upper halfspace near \(C\). Assume that \(M\) has a finite number \(n\geq0\) of vertical annular ends. We index with \(i\in I\) the positive ends and with \(j\in J\) the negative ones. 
	
	If \(M\) is compact, then \(M\) is contained in the upper halfspace.
	
	If \(M\) is not compact, then either \(M\) is contained in the upper halfspace or there is a simple loop $\gamma\subset M\cap(\mathcal{H}\setminus\bar{D})$, such that $\gamma$ generates \(\Pi_1(\mathcal{H}\setminus\bar{D})\). In the latter case, it holds:
	\begin{itemize}
		\item \(|D|\leq\sum_{i\in I}\mathcal{W}_i-\sum_{j\in J}\mathcal{W}_j\) when \(\sum_{i\in I}\mathcal{W}_i-\sum_{j\in J}\mathcal{W}_j\geq0\).
		\item \(|D|\geq\sum_{j\in J}\mathcal{W}_j-\sum_{i\in I}\mathcal{W}_i\) when \(\sum_{i\in I}\mathcal{W}_i-\sum_{j\in J}\mathcal{W}_j<0\);
	\end{itemize}
	where, for each \(i\in I\cup J\), $\mathcal{W}_i$ is the mass of the \(i\)-th end.
\end{theorem}

Now, consider the case of non-negative balance. As a consequence of Theorem\,\,\ref{twocmc} we get that, in presence of ends, the balance \(\sum_{i\in I}\mathcal{W}_i-\sum_{j\in J}\mathcal{W}_j\) is \textit{strictly positive}. This means that, given a surface \(M\) satisfying the hypothesis of Theorem\,\,\ref{twocmc} and such that:
\[\sum_{i\in I}\mathcal{W}_i=\sum_{j\in J}\mathcal{W}_j,\]
it holds that \(M\) must be compact, i.e. \(I\cup J=\emptyset\). This can also be seen as a partial result in the direction of the original question of Rosenberg and Sa Earp for cmc surfaces \cite{saearpconvex}, whose analogue for Weingarten surfaces was mentioned in the beginning: let \(C\) be a circle in the plane $\mathcal{H}$ and let \(M\) be a properly embedded cmc (or linear Weingarten, with \(a,b>0\)) surface, with $\partial M\equiv C$ and \(M\) transverse to $\mathcal{H}$ along \(C\), having a finite number of annular ends, which are all vertical. Must \(M\) be rotational? The previous discussion tells us that, when the total mass of the positive ends is equal to the total mass of the negative ends, then \(M\) is rotational (and it is in fact compact).

\phantom{ciao}

{\bf Acknowledgments}. This study was partially supported by INdAM-GNSAGA and PRIN-2022AP8HZ9. I am very grateful to Barbara Nelli for first proposing the problem and for many stimulating discussions, which were fundamental in the progress of the present work.


\begin{thebibliography}{10}
	
	\bibitem{espinar}
	{\sc J.~A. Aledo, J.~M. Espinar, and J.~A. G\'alvez}, {\em The {C}odazzi
		equation for surfaces}, Adv. Math., 224 (2010), pp.~2511--2530.
	
	\bibitem{britoboundarycircle}
	{\sc F.~G.~B. Brito and R.~Sa~Earp}, {\em On the structure of certain
		{W}eingarten surfaces with boundary a circle}, Ann. Fac. Sci. Toulouse Math.
	(6), 6 (1997), pp.~243--255.
	
	\bibitem{chernsphere}
	{\sc S.-s. Chern}, {\em Some new characterizations of the {E}uclidean sphere},
	Duke Math. J., 12 (1945), pp.~279--290.
	
	\bibitem{chernwsurfaces}
	\leavevmode\vrule height 2pt depth -1.6pt width 23pt, {\em On special
		{$W$}-surfaces}, Proc. Amer. Math. Soc., 6 (1955), pp.~783--786.
	
	\bibitem{tenenblat}
	{\sc A.~V. Corro, W.~Ferreira, and K.~Tenenblat}, {\em Ribaucour
		transformations for constant mean curvature and linear {W}eingarten
		surfaces}, Pacific J. Math., 212 (2003), pp.~265--296.
	
	\bibitem{Delaunay1841}
	{\sc C.~E. Delaunay}, {\em Sur la surface de révolution dont la courbure
		moyenne est constante.}, Journal de Mathématiques Pures et Appliquées,
	(1841), pp.~309--314.
	
	\bibitem{britosaearpmeeks}
	{\sc R.~Sa~Earp, F.~G.~B.~Brito, W.~H. Meeks, III, and H.~Rosenberg}, {\em Structure
		theorems for constant mean curvature surfaces bounded by a planar curve},
	Indiana Univ. Math. J., 40 (1991), pp.~333--343.
	
	\bibitem{fernandezgalvezmira}
	{\sc I.~Fern\'andez, J.~A. G\'alvez, and P.~Mira}, {\em Quasiconformal {G}auss
		maps and the {B}ernstein problem for {W}eingarten multigraphs}, Amer. J.
	Math., 145 (2023), pp.~1887--1921.
	
	\bibitem{fernandezweingarten}
	{\sc I.~Fern\'andez and P.~Mira}, {\em Elliptic {W}eingarten surfaces:
		singularities, rotational examples and the halfspace theorem}, Nonlinear
	Anal., 232 (2023), pp.~Paper No. 113244, 27.
	
	\bibitem{galvezlinear}
	{\sc J.~A. G\'alvez, A.~Mart\'inez, and F.~Mil\'an}, {\em Linear {W}eingarten
		surfaces in {$\Bbb R^3$}}, Monatsh. Math., 138 (2003), pp.~133--144.
	
	\bibitem{galvezmira}
	{\sc J.~A. G\'alvez and P.~Mira}, {\em Rotational symmetry of {W}eingarten
		spheres in homogeneous three-manifolds}, J. Reine Angew. Math., 773 (2021),
	pp.~21--66.
	
	\bibitem{hartmanwinter}
	{\sc P.~Hartman and A.~Wintner}, {\em Umbilical points and {$W$}-surfaces},
	Amer. J. Math., 76 (1954), pp.~502--508.
	
	\bibitem{Korevaar1989TheSO}
	{\sc N.~J. Korevaar, R.~B. Kusner, and B.~Solomon}, {\em The structure of
		complete embedded surfaces with constant mean curvature}, Journal of
	Differential Geometry, 30 (1989), pp.~465--503.
	
	\bibitem{meeks}
	{\sc W.~H. Meeks~III}, {\em {The topology and geometry of embedded surfaces of
			constant mean curvature}}, Journal of Differential Geometry, 27 (1988),
	pp.~539 -- 552.
	
	\bibitem{someproblems}
	{\sc T.~Rassias and R.~Sa~Earp}, {\em Some problems in analysis and geometry}, in Complex Analysis in Several Variables,
	Hadronic Press, Florida, 01 1999, pp.~111--122.
	
	\bibitem{saearpconvex}
	{\sc H.~Rosenberg and R.~Sa~Earp}, {\em Some structure theorems for complete
		constant mean curvature surfaces with boundary a convex curve}, Proc. Amer.
	Math. Soc., 113 (1991), pp.~1045--1053.
	
	\bibitem{saearpweingarten}
	\leavevmode\vrule height 2pt depth -1.6pt width 23pt, {\em The geometry of
		properly embedded special surfaces in {${\bf R}^3$}, e.g., surfaces
		satisfying {$aH+bK=1$}, where {$a$} and {$b$} are positive}, Duke Math. J.,
	73 (1994), pp.~291--306.
	
	\bibitem{saearpfrench}
	{\sc R.~Sa~Earp and E.~Toubiana}, {\em Classification des surfaces de type
		{D}elaunay}, Amer. J. Math., 121 (1999), pp.~671--700.
	
\end{thebibliography}
\end{document}